\documentclass[12pt,leqno]{article} 

\usepackage{amsmath} \usepackage{amscd} 
\usepackage{amssymb,pictex} 
\usepackage{amsfonts}
\usepackage{amsthm}
\usepackage{graphics,graphicx,psfrag,color}
\usepackage{verbatim}
\newtheorem{satz}{Theorem}
\newtheorem{cor}[satz]{Corollary}
\newtheorem{rem}[satz]{Remark}
\newtheorem{defi}[satz]{Definition}
\newtheorem{lem}[satz]{Lemma}
\newtheorem{example}[satz]{Example} 
\newtheorem{examples}[satz]{Examples} 
\newtheorem{prop}[satz]{Proposition}

\def\cx{{\mathcal X}}

\def\N{{\mathbb N}}
\def\R{{\mathbb R}} 
\def\Rd{{\mathbb R}^d}

\def\dist{{\rm dist \, }}

\def\ud {{\rm d}}
\def\nat{{\mathbb N}}
\def\real{{\mathbb R}}

\newcommand{\calk}{{\mathcal K}}
\DeclareMathOperator{\supp}{supp}

\newcommand{\be}{\begin{equation}}
\newcommand{\ee}{\end{equation}}
\newcommand{\beq}{\begin{eqnarray}}
\newcommand{\beqq}{\begin{eqnarray*}}
\newcommand{\eeq}{\end{eqnarray}}
\newcommand{\eeqq}{\end{eqnarray*}}

\title{Properties of Kondratiev spaces} 

\author{Stephan Dahlke\thanks{The work of this author has been supported by Deutsche Forschungsgemeinschaft (DFG), grant Da360/22--1.},
\quad Markus Hansen \thanks{The work of this author has been supported by the ERC, Starting Grant 306274 (HDSPCONTR).},
\quad Cornelia Schneider \thanks{The work of this author has been supported by Deutsche Forschungsgemeinschaft (DFG), grant SCHN 1509/1-1.}\\ 
\quad \& \:  Winfried Sickel} 

\begin{document}

\maketitle

\begin{abstract}  
In this paper, we investigate Kondratiev spaces on domains of polyhedral type. In particular, we will be concerned 
with necessary and sufficient conditions for  continuous and compact embeddings and in addition we shall deal 
with pointwise multiplication in these spaces. 
\end{abstract}



\tableofcontents


\section{Introduction}


Since the midsixties  scales of weighted Sobolev spaces have become
popular in the
study of regularity of solutions to elliptic PDEs
on polygonal and polyhedral domains in $\R^2$ and $\R^3$, respectively.
Let us mention the pioneering work of Kondratiev \cite{Kon1}, \cite{Kon2},
see also the survey of Kondratiev and Oleinik \cite{Kon3}.
Later on these types of spaces, partly more general, have been considered by  many
people. Let us mention just a few:
Babu\v{s}ka, Guo \cite{BabuskaGuo}, Bacuta, Mazzucato, Nistor, Zikatanov
\cite{BMNZ},
Dauge \cite{Dauge}, Kozlov, Maz'ya, Rossmann \cite{KMR}, \cite{KMR2},
Kufner, S\"andig \cite{Ku}, Maz'ya, Rossmann \cite{MR},  Mazzucato, Nistor
\cite{NistorMazzucato}, and Nazarov, Plamenevskii \cite{NP}.
Whereas in the mentioned references  the weight was always chosen to be
a power of the distance to the singular set of the boundary, there are also
publications dealing with the weight being a power of
the distance to the whole boundary. We refer e.g. to Kufner, S\"andig
\cite{Ku},
Triebel \cite[3.2.3]{T78} and Lototsky \cite{Lo}.

Kondratiev spaces  provide a very powerful tool in the context of the
qualitative theory of elliptic and parabolic PDEs, especially on
nonsmooth domains.  In particular, on domains with edges and corners,  these
nonsmooth parts of the boundary induce singularities for the solution and
its derivatives.  By means of Kondratiev spaces it is possible to
describe very precisely the behaviour of these singularities. Moreover,
these specific smoothness spaces allow for certain shift theorems in the
following  sense. Suppose that we are given a second order  elliptic
differential equation  on a polygonal or polyhedral domain. Then, under
certain conditions on the coefficients and on the domain, it turns out that  if
the right--hand side has smoothness $m-1$ in the scale of Kondratiev
spaces, then the solution $u$ of the PDE has smoothness $m+1$. We refer   to    
\cite{BMNZ} and {particularly} to  \cite{MR} for further
information. While for  smooth domains similar statements also hold for
classical smoothness spaces such as Sobolev spaces, the situation is
completely different on the nonsmooth domains we are concerned with here.
{In this case} the singularities at the boundary diminish the Sobolev
regularity. Let us {in this context} mention the famous $H^{3/2}$-theorem proved by Jerison and Kenig   
\cite{JK1}, which says that for the
Poisson equation  there exist Lipschitz domains and right-hand sides $f\in
C^{\infty}$ such that the smoothness  of the corresponding solution is
limited by $3/2$.

The above remarks reflect that Kondratiev spaces have been shown to be an
indispensable tool in the theory of elliptic equations,
in particular, on non-convex polyhedral domains.
In this paper, as a first step, we intend to give a survey on what is
known about
these classes
with respect to continuous and compact embeddings.
As a second step we would like to understand
the mappings $u \mapsto u^n$, $n \in \N$, in the framework of these {scales}.
To do this we will allow a greater generality.
{Moreover}, we will give the final answer under which conditions
Kondratiev spaces form algebras with respect to pointwise multiplication.

There is also an interesting   relationship of Kondratiev spaces {with}
important  issues in numerical analysis.  As is well-known, the
approximation order that can be achieved by adaptive and other nonlinear
methods usually depends on the regularity of the exact solutions in scales
of Besov spaces \cite{DD,DahlkeSickel,DahlkeSickel2,Hansen}. Since there exist a lot of embeddings of
Kondratiev spaces into Besov spaces, cf.  \cite{Hansen, Hansen2},
Besov regularity estimates can very often be traced back to regularity
questions in Kondratiev spaces. Therefore, the results presented in this
paper  will be used in a follow-up paper \cite{DHSS1} in order to look at Besov
regularity of solutions to nonlinear elliptic partial differential equations, e.g.  
\beqq
 -\Delta u (x) + u^n (x) & = &  f(x)\, , \qquad x \in \Omega, \quad n>2,
\\
u(x) &= & 0 \, , \qquad \quad \  x \in \partial \Omega\, .
\eeqq

The paper is organized as follows.
In Section \ref{Kond} we shall give the definition of the scales
$\calk^m_{a,p}(\Omega,M)$.
There we also discuss in detail in which types of domains we are
interested in, cf. Subsection \ref{domains}.
The next section is {then} devoted to the study of necessary and sufficient
conditions for continuous and compact embeddings of Kondratiev spaces.
In Section \ref{pm} we discuss pointwise multipliers for Kondratiev
spaces in great detail. Whereas other parts of the paper have the character of a survey, 
the contents of this section are completely new.
Firstly, we  investigate under which conditions on the parameters
$m,p$, and $a$
a space  $\calk^m_{a,p}(\Omega,M)$ forms an algebra with respect to pointwise
multiplication.
Secondly, under certain conditions on the parameters, we also deal 
with the more general case of products of the form $\calk^{m_1}_{a_1,p_1}(\Omega,M) \, \cdot\,  \calk^{m_2}_{a_2,p_2}(\Omega,M)$.
\\
In almost all cases  the following strategy {is used}.
In a first step we deal with the {corresponding} problem for simplified Kondratiev spaces defined 
on $\calk^{m}_{a,p} (\R^d,\R^\ell_*)$. 
Afterwards, {using} linear and continuous extension operators, we  extend the obtained results 
to Kondratiev spaces defined on smooth cones, nonsmooth cones and specific dihedral domains, see Cases I-III on pages \pageref{smoothconekondratiev}-\pageref{Kondratievdihedral} 
below.
In a third step, by making use of a simple decomposition, we are able to handle 
Kondratiev spaces defined on polyhedral cones, see Case IV on page \pageref{def-norm-poly-cone}.
Furthermore, the   decomposition from Lemma \ref{deco} allows us to extend everything to so-called domains of polyhedral 
type. Not all interesting examples are covered by this notion.
Therefore, we proceed  in the last step with a further generalization, based on Lemma \ref{deco3}.
This yields the results for Kondratiev spaces  defined on  pairs $(D,M)$  of generalized polyhedral type.

Some final remarks concerning the choice of our weighted Sobolev spaces are in order.  
In our setting, see Definition \ref{def-kondratiev} of Subsection \ref{sub1}, all derivatives that 
occur are weighted by some power of  the distance to the singularity set, where the power depends on 
the order $\alpha$ of the corresponding derivative. This is of course not the only possible choice. Indeed, 
several authors worked with the scale  $J_{\gamma}^m$, where the power does {\em not} depend on $\alpha$. 
Let us just mention the work   of Babu\v{s}ka and Guo [3] and Costabel, Dauge and Nicaise \cite{CDN2}  (this list is clearly 
not complete).  It is sometimes claimed that the scale $J_{\gamma}^m$ is more versatile in order to describe the global 
regularity of solutions of PDEs. And indeed, these spaces have  the  advantage that for large enough $m$  they may 
contain all polynomials, which is not true in our case. Consequently,  based on (intersections of) these spaces, Babu\v{s}ka and 
Guo \cite{BabuskaGuo} and Guo  \cite{Guo}  have been able to show exponential convergence of  $hp$-versions of Finite
 Element Methods. However, for our purposes, the Kondratiev scale as introduced in Section 2 is more suitable for 
several reasons. In particular,  (complex) interpolation of these spaces is  much simpler and the desired embeddings into scales of 
Besov spaces also arise very naturally.


\section{Kondratiev spaces}
\label{Kond}


Let us start by collecting some general notation used throughout the paper.

As usual, $\N$ stands for the set of all natural numbers, $\N_0=\mathbb N\cup\{0\}$, and 
$\Rd$, $d\in\N$, is the $d$-dimensional real Euclidean space with $|x|$, for $x\in\Rd$, denoting the Euclidean norm of $x$. 
Let $\N_0^d$, where $d\in\N$, be the set of all multi-indices, $\alpha = (\alpha_1, \ldots,\alpha_d)$ with 
$\alpha_j\in\N_0$ and $|\alpha| := \sum_{j=1}^d \alpha_j$. 

Furthermore, $B_{\varepsilon}(x)$  is the open ball 
of radius $\varepsilon >0$ centred at $x$. 

We denote by  $c$ a generic positive constant which is independent of the main parameters, but its value may change from line to line. 
The expression $A\lesssim B$ means that $ A \leq c\,B$. If $A \lesssim
B$ and $B\lesssim A$, then we write $A \sim B$.  

Given two quasi-Banach spaces $X$ and $Y$, we write $X\hookrightarrow Y$ if $X\subset Y$ and the natural embedding is bounded.

A domain $\Omega$ is an open bounded set in $\R^d$. 
Let $L_p(\Omega)$, $1\leq p\leq \infty$, be the Lebesgue spaces on $\Omega$ as usual.  
Furthermore, for $m \in \N$ and  $1\le p \leq \infty$, we denote by
 $W^{m}_p(\Omega)$ the standard Sobolev space  on the domain $\Omega$ equipped with the norm
\[
\|\, u \, |W^m_{p}(\Omega)\| := \Big(\sum_{|\alpha|\leq m} \int_\Omega |\partial^\alpha u(x)|^p\,dx\Big)^{1/p}\, 
\]
(with the usual modification if $p=\infty$). 
If $p=2$ we shall also write $H^m(\Omega)$ instead of $W^m_2(\Omega)$.


\subsection{Definition and basic properties}
\label{sub1}


\begin{defi} \label{def-kondratiev}
Let $\Omega$ be a  domain in $\R^d$  and let  $M$ be a non-trivial closed subset of its boundary ${\partial \Omega}$.
Furthermore, 
let  $m\in\N_0$ and $a\in\R$. 
We put 
\be\label{gewicht} 
\rho(x):= \min\{1, \dist(x,M)\} \, , \qquad x \in \Omega\, .
\ee
{\rm (i)} Let $1\le p < \infty$.
We define the Kondratiev spaces $\calk^m_{a,p}(\Omega, M)$ as the collection of all measurable functions which 
admit $m$ weak derivatives in $\Omega$ satisfying
\[
\|u|\calk^m_{a,p}(\Omega,M)\| := \Big(\sum_{|\alpha|\leq m}\int_\Omega
|\rho(x)^{|\alpha|-a}\partial^\alpha u(x)|^p\,dx\Big)^{1/p}
<\infty\,.
\]
{\rm (ii)} The space $\calk^m_{a,\infty}(\Omega, M)$ is the collection of all measurable functions which 
admit $m$ weak derivatives in $\Omega$ satisfying
\[
\|u|\calk^m_{a,\infty}(\Omega,M)\| := \sum_{|\alpha|\leq m} \|\, \rho^{|\alpha|-a} \partial^\alpha u \, |L_\infty (\Omega)\|
<\infty\,.
\]
\end{defi}

\begin{rem} 
\rm
{\rm (i)}
Many times  the set $M$ will be the {\em singularity set} $S$ of the domain $\Omega$,  i.e., 
the set of all points $x \in \partial \Omega$ such that  for any $\varepsilon >0$  the set
$\partial \Omega \cap B_{\varepsilon}(x)$ is not smooth. 
\\
{\rm (ii)} We will not distinguish spaces which differ by an equivalent norm.
\\
\end{rem}

\subsection*{Basic properties}

We collect basic properties of Kondratiev spaces that will be useful in what follows.  

\begin{itemize}
\item
$\calk^m_{a,p}(\Omega,M)$ is a Banach space, see \cite{KO1}, \cite{KO2}.

\item The scale of Kondratiev spaces is monotone in $m$ and $a$, i.e.,
\begin{equation}\label{kondr-emb}
\calk^m_{a,p}(\Omega,M) \hookrightarrow \calk^{m'}_{a,p}(\Omega,M) \qquad \mbox{and} \qquad 
\calk^m_{a,p}(\Omega,M) \hookrightarrow \calk^m_{a',p}(\Omega,M)
\end{equation}
if $m' < m$ and $a' < a$.

\item Regularized distance function: There exists a function $\tilde{\varrho}: ~ \overline{\Omega} \to [0,\infty)$, which is infinitely often 
differentiable in $\Omega$, and   
positive constants $A,B,C_\alpha$ such that 
\[
 A \, \rho (x) \le \tilde{\varrho}(x)\le B\, \rho (x) \, , \qquad x \in \Omega \, ,
\]
and, for all $\alpha \in \N_0^d$, 
\[
\Big| \partial^\alpha \, \tilde{\varrho}(x) \Big|\le C_\alpha \rho^{1-|\alpha|} (x)\, , \qquad x \in \Omega \, ,
\]
see Stein \cite[Thm.~VI.2.2]{St70} (the construction given  there is valid for arbitrary closed subsets of $\R^d$).

\item \label{isom-prop}
By using the previous item and replacing $\rho$ by $\tilde{\varrho}$ in the norm of $\calk^m_{a,p}(\Omega,M)$ 
one can prove the following.\\
Let $b \in \R$.
Then the  mapping $T_b : ~ u \mapsto \tilde{\varrho}^b \, u $ yields an isomorphism  of $\calk^m_{a,p}(\Omega,M)$ onto $\calk^m_{a+b,p}(\Omega,M)$. 

\item 
Let $a\ge 0$. Then $\calk^m_{a,p}(\Omega,M) \hookrightarrow L_{p}(\Omega)$.

\item 
A function $\psi:~ \Omega \to \R $ such that the ordinary derivatives $\partial^\alpha \psi$ are  continuous functions on $\Omega$
for all $\alpha$, $|\alpha|\le m$, 
\[
 \| \psi|C^m (\Omega)\|:= \max_{|\alpha|\le m} \, \sup_{x \in \Omega} |\partial^\alpha \psi (x)| < \infty\, , 
\]
is a pointwise multiplier for $\calk^m_{a,p}(\Omega,M)$, i.e., $\psi \,\cdot \,  u \in \calk^m_{a,p}(\Omega,M) $ for all 
$u \in \calk^m_{a,p}(\Omega,M)$.

\item Let $1\le p < \infty$. For a given domain $\Omega$ define
$$\Omega_\delta=\{x\in \Omega: ~ \dist (x,M)>\delta\},\quad\delta>0.$$
We assume that $\Omega_\delta$ has the segment property for all $\delta\leq\delta_0$ for some sufficiently small $\delta_0>0$; further suppose
\[
\lim_{\delta \downarrow 0}\,  |\{x\in \Omega: ~ \dist (x,M)\leq\delta\}|
	=\lim_{\delta \downarrow 0}\,|\Omega\setminus\Omega_\delta|=0\, .
\]
Then $C_*^\infty (\Omega,M)=\{u|_\Omega:u\in C_0^\infty(\R^d\setminus M)\}$ is a dense subset of
$\calk^m_{a,p}(\Omega,M)$.

This is the weighted counterpart of \cite[Theorem 4.7]{EdmundsEvans}. We shall sketch the argument: Due to the assumption on $\Omega$ and $M$, for any given fixed $u\in\calk^m_{a,p}(\Omega,M)$ and
$\varepsilon>0$ by the absolute continuity of the Lebesgue integral we can find a $\delta>0$ such that
$$\sum_{|\alpha|\leq m}\int_{\Omega\setminus\Omega_\delta}
	|\rho(x)^{|\alpha|-a}\partial^\alpha u(x)|^p\,dx
	<\varepsilon.$$
On the other hand, we have
$$\sum_{|\alpha|\leq m}\int_{\Omega_\delta}
	|\rho(x)^{|\alpha|-a}\partial^\alpha u(x)|^p\,dx
	\sim\|u|W^m_p(\Omega_\delta)\|^p,$$
with constants depending only on $\delta$, $m$, $a$ and $p$. The result now follows by applying the unweighted density result \cite[Theorem 4.7]{EdmundsEvans} to the domain $\Omega_\delta$.
\end{itemize}


\subsection{Domains of polyhedral type}
\label{domains}


In the sequel, we will mainly be interested in the case that  $d$ is either $2$ or $3$ and that  $\Omega$ is a bounded domain of polyhedral type. 
The precise definition will be given below in Definition \ref{standard}. 
Essentially, we will consider domains for which the analysis of the associated Kondratiev spaces can be reduced to the 
following four  basic cases:
\begin{itemize}
 \item Smooth cones;
 \item Specific nonsmooth cones;
 \item Specific dihedral domains;
 \item Polyhedral cones.
\end{itemize}
Let $d\ge 2$.
As usual, an infinite  smooth  cone with vertex at the origin is the set
\[
 K:= \{x\in \Rd: \: 0 < |x|< \infty\, , \: x/|x| \in \Omega\}\, , 
\]
where $\Omega$ is a subdomain of the unit sphere $S^{d-1}$ in $\Rd$ with $C^\infty$ boundary.\\
{~}\\

\noindent 
\begin{minipage}{0.4\textwidth}
{\bf Case I:} {\em Kondratiev spaces on smooth cones.}  
Let $K'$ be an infinite  smooth cone contained in $\R^d$ with vertex at the origin 
which is  rotationally invariant with respect to the axis $\{(0,\ldots,0, x_d):~ x_d \in \R\}$. Let  $M:=\{0\}$.
Then we define the truncated cone $K$  by \\
 \[
K:= K'\cap B_1(0)
\] 
and put
\end{minipage}\hfill \begin{minipage}{0.5\textwidth}
\includegraphics[width=6cm]{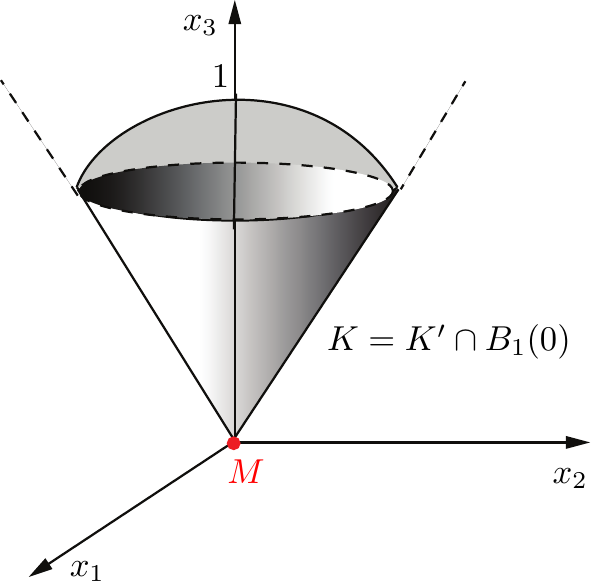}
\end{minipage}
\begin{equation} \label{smoothconekondratiev}
\|u|\calk^m_{a,p}(K,M)\| :=
 \Big(\sum_{|\alpha|\leq m}\int_{K}|\,  |x|^{|\alpha|-a}\partial^\alpha u(x)|^p\,dx\Big)^{1/p} \, .
\end{equation}
There is still one degree of freedom in the choice of the smooth cone, namely the opening angle $\gamma \in (0,\pi)$ of the cone. 
Since this will be unimportant in what follows we will not indicate this in the notation.
Observe that $M$ is just a part of the singular set of the boundary of {the truncated cone} $K$.
\\

\noindent
\begin{minipage}{0.45\textwidth}
{\bf Case II:} {\em Kondratiev spaces on specific nonsmooth cones}. 
Let again $K'$ denote a rotationally  symmetric cone as described in  {\bf Case I} with opening angle $\gamma \in (0,\pi)$. 
Then we define the specific nonsmooth cone $P$  by  $P=K'\cap I$, where $I$ denotes the unit cube
\be\label{ws-01} 
I :=\{x \in \R^d:~0<x_i<1,~i=1,\ldots,d\}.
 \ee
 \end{minipage}\hspace{0.5cm} \begin{minipage}{0.5\textwidth}
\quad \includegraphics[width=6.5cm]{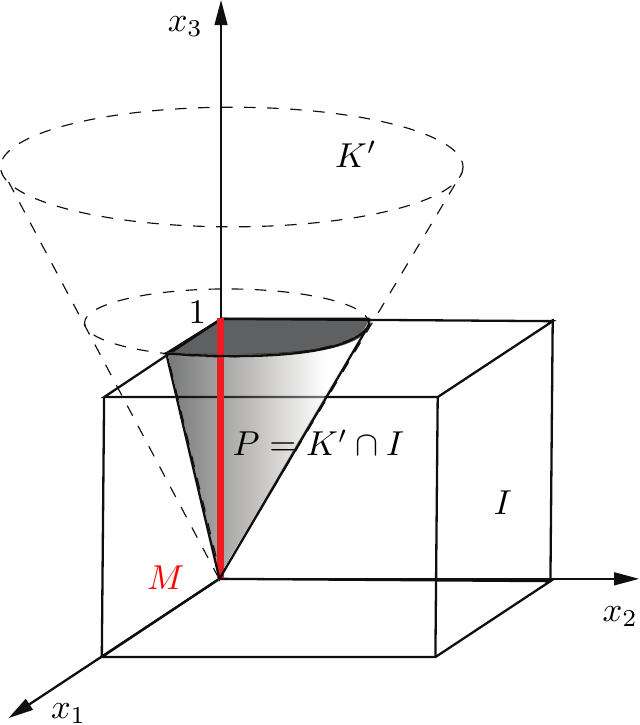}
\end{minipage}
\\
{~}\\
\noindent
In this case, we choose  $\Gamma : = \{x\in \R^d:~ x=(0,\ldots,0, x_d),~0 \le x_d \le 1\}$ and define
\begin{equation} 
\|u|\calk^m_{a,p}(P,\Gamma)\|:=
\Big(\sum_{|\alpha|\leq m}\int_{P}|\,  \rho(x)^{|\alpha|-a}\partial^\alpha u(x)|^p\,dx\Big)^{1/p}\, , 
\end{equation}
where $\rho(x)$ denotes the distance of $x$ to $\Gamma$, i.e.,  $\rho (x) = |(x_1,\ldots,x_{d-1})|.$
Again the opening angle $\gamma$ of the cone $K'$ will be of no importance.
Also in this case the set $\Gamma$ is a proper subset of the singular set of $P$.
\\

\noindent
\begin{minipage}{0.5\textwidth}
{\bf Case III:} {\em Kondratiev spaces on specific dihedral domains}.  
\\
Let $1\le \ell < d$ and let $I$ be the unit cube defined in \eqref{ws-01}.
For $x \in \R^d$ we write 
\[
x=(x',x'') \in \R^{d-\ell}\times \R^{\ell},
\] 
where $x':= (x_1, \ldots \, , x_{d-\ell})$ and
\mbox{$x'':= (x_{d-\ell + 1}, \ldots \, , x_{d})$}. Hence 
$I = I' \times I''$ with the obvious interpretation. Then we choose
 \end{minipage}\hfill \begin{minipage}{0.4\textwidth}
\includegraphics[width=6cm]{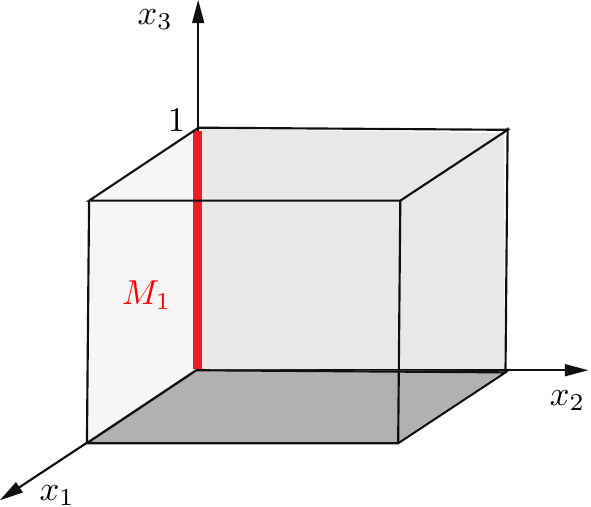}
\end{minipage}
\be\label{ws-15}
M_\ell := \{x \in \partial I:~x_1 = \ldots = x_{d-\ell }=0\} 
\ee
and define
\begin{equation} \label{Kondratievdihedral}
\|u|\calk^m_{a,p}(I, M_\ell) \|
:= \Big(\sum_{|\alpha|\leq m}\int_{I}|\,  |x'|^{|\alpha|-a}\partial^\alpha u(x)|^p\,dx\Big)^{1/p}\, .
\end{equation}

\noindent
\begin{minipage}{0.5\textwidth}
{\bf Case IV:} {\em Kondratiev spaces on  polyhedral cones}. 
Let $ K'$ 
be an infinite cone in $\R^3$, more exactly, in the half space $x_3 >0$ with vertex at the origin. We assume that the boundary $\partial K'$ 
consists of the vertex $x=0$, the edges (half lines) $M_1, \ldots \, , M_n$, and smooth faces $\Gamma_1, \ldots , \Gamma_n $.
This means $\Omega := K' \cap S^{2}$ is a domain of polygonal type on the unit sphere with sides 
$\Gamma_k\cap S^2$.
We put 
\[
Q:= K' \,  \cap \, \{x \in \R^3: ~ 0 < x_3 < 1\}\, .
\]
\end{minipage}\hfill \begin{minipage}{0.45\textwidth}
\quad \includegraphics[width=6cm]{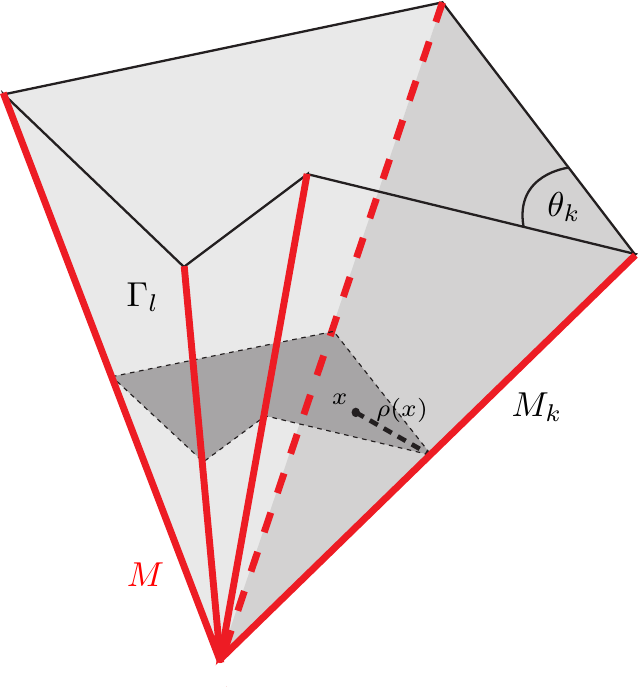}
\end{minipage}\\
{~}\\
 In this case, we choose  $M: = (M_1 \cup \ldots \cup M_n)\cap\overline{Q}$ and define
\begin{equation} \label{def-norm-poly-cone}
\|u|\calk^m_{a,p}(Q,M)\|:=
\Big(\sum_{|\alpha|\leq m}\int_{Q}|\,  \rho(x)^{|\alpha|-a}\partial^\alpha u(x)|^p\,dx\Big)^{1/p}\, ,
\end{equation}
where $\rho(x)$ denotes the distance of $x$ to $M$. 
\\

Based on these four cases, we define the specific domains we will be concerned with in this paper.

\begin{defi} \label{standard} 
Let  $D$ be a domain in $\R^d$, $d\geq 2$, with singularity set $S$. 
Then $D$ is of {\em polyhedral type}, if there exists a finite covering $(U_i)_i$ of bounded open sets 
such that 
$$ 
\overline{D}
	\subset \Big(\bigcup_{i \in \Lambda_1} U_i\Big)
		\cup\Big( \bigcup_{j \in \Lambda_2} U_{j} \Big) 
		\cup\Big( \bigcup_{k \in \Lambda_3} U_{k}\Big)
		\cup\Big( \bigcup_{\ell \in \Lambda_4} U_{\ell}\Big)\,,
$$
where
\begin{itemize}
\item[i)]  $i \in \Lambda_1$ if $U_i$ is a ball and  $\overline{U_i} \cap S = \emptyset$.

\item[ii)] $j\in \Lambda_2$ if there exists a $C^{\infty}$-diffeomorphism 
$\eta_{j}~:~\overline{U_{j}}\longrightarrow \eta_{j}(\overline{U_{j}})\subset \R^d$ such that  
$\eta_{j}(U_{j} \cap D)$ is the  smooth cone $K$ as described in {\bf Case I}. 
Moreover, we assume that  for all $x \in U_{j} \cap D$ the distance to $S$ 
is equivalent to the distance to the point $x^{j} := \eta_{j}^{-1}(0).$

\item[iii)] $k \in \Lambda_3$ if there exists a $C^{\infty}$-diffeomorphism  
$\eta_{k}~:~\overline{U_{k}} \longrightarrow \eta_{k}(\overline{U_{k}})\subset  \R^d$ 
such that $\eta_{k}(U_{k}\cap D)$ is the nonsmooth   cone $P$ as described in {\bf Case II}. 
Moreover, we assume that for all $x\in U_{k}\cap D$ the distance to  
$S$ is equivalent to the distance to the set $\Gamma^{k}:= \eta_{k}^{-1}(\Gamma)$. 

\item[iv)]$ \ell \in \Lambda_{4}$  if there exists a $C^{\infty}$-diffeomorphism 
$\eta_{\ell}: \overline{U_{\ell}} \longrightarrow  \eta_{\ell}(\overline{U_{\ell}}) \subset\R^d$ such that
$\eta_{\ell}(U_{\ell}\cap D)$  is a specific dihedral domain  as described in {\bf Case III}. 
Moreover, we assume that for all $x\in U_{\ell}\cap D$ the distance to  $S$ is equivalent to the distance to 
the set $M^{\ell}:= \eta_{\ell}^{-1}(M_{n})$ for some $n \in \{1, \ldots \, , d-1\}$.
\end{itemize}

\noindent
In particular,  when $d=3$ we permit another type of subdomain: then 
$$ 
\overline{D}
	\subset \Big(\bigcup_{i \in \Lambda_1} U_i\Big)
		\cup\Big( \bigcup_{j \in \Lambda_2} U_{j} \Big) 
		\cup\Big( \bigcup_{k \in \Lambda_3} U_{k}\Big)
		\cup\Big( \bigcup_{\ell \in \Lambda_4} U_{\ell}\Big)
		\cup\Big( \bigcup_{m \in \Lambda_5} U_{m}\Big)\, ,
$$
where
\begin{itemize}
\item[v)]$ m \in \Lambda_{5}$  if there exists a $C^{\infty}$-diffeomorphism 
$\eta_{m}: \overline{U_{m}} \longrightarrow  \eta_{m}(\overline{U_{m}}) \subset\R^3$ such that
$\eta_{m}(U_{m}\cap D)$  is a polyhedral cone  as described in {\bf Case IV}. 
Moreover, we assume that for all $x\in U_{m}\cap D$ the distance to  $S$ is equivalent to the distance to 
the set $M'_m := \eta_{m}^{-1}(M)$.

\end{itemize}
\end{defi}

\begin{rem}\label{rem-def-domain}
\rm
(i)
In the literature many different types of polyhedral domains are considered. As will be discussed below, in our context only 
the Cases I and III are essential. Therefore, our definition coincides with the one of Maz'ya and Rossmann \cite[4.1.1]{MR}. 
Further variants can be found in 
Babu\v{s}ka,  Guo \cite{BabuskaGuo},  Bacuta, Mazzucato, Nistor, Zikatanov \cite{BMNZ} 
and Mazzucato, Nistor \cite{NistorMazzucato}.
\\
(ii)
While the types of polyhedral domains coincide, in \cite{MR} more general weighted Sobolev spaces on those polyhedral 
domains are discussed. Our spaces $\calk^m_{a,p}(D,S)$  coincide with the classes $V^{\ell,p}_{\beta,\delta} (D)$ if $m= \ell$, 
\[
\beta = (\beta_1, \ldots\, , \beta_{d'})= (\ell-a, \ldots \, , \ell-a) \qquad \mbox{and}\qquad    \delta = (\delta_1, \ldots\, , \delta_{d}) = 
( \ell - a, \ldots \, ,\ell - a )\, .
\] 
For the meaning of $d$ and $d'$ we refer to \cite[4.1.1]{MR}.
\end{rem}

The definition contains some redundancies. 
It is rather apparent that Cases II and III essentially coincide. A little less obvious (though via still quite basic geometric 
arguments) it can be seen that also Case IV can be reduced to Cases I and II (see Lemma \ref{deco2} below). However, this simple domain 
covering and the resulting norm decomposition are not applicable to every situation: the  method does not allow the usage of 
a resolution of unity with compactly supported functions within the subdomains as discussed in Lemma \ref{deco} -- while a 
finite cover can be given, the compact supports of the functions from the resolution of unity prevent from getting arbitrarily 
close to the vertex of the polyhedral cone. Alternatively, one has to specifically include  a neighbourhood of that vertex, 
on which the distance function is neither equivalent to the distance to an edge nor to the distance to the vertex. Fortunately, 
it turns out that the results presented in this article can be proven without the decomposition result from Lemma \ref{deco}.
One situation where the usage of a resolution of unity as in Lemma \ref{deco} is necessary arises when considering extension 
operators. This will be discussed in a forthcoming publication \cite{Hansen3}.
Despite this redundancy we still decided to include polyhedral cones in the above definition since they 
represent an important special case and, moreover, a number of results can be proved directly for such cones.  
This makes a reduction to other cases unnecessary and the presentation itself more  accessible.

\begin{examples}

Examples of numerically interesting domains, which are of polyhedral type and are therefore covered by our 
investigations, are $L$-shaped domains and 'the donut' in 2D. Moreover,  
in 3D one could mention the Fichera corner and the {icosahedron}.\\

\begin{minipage}{0.5\textwidth}
\includegraphics[width=5cm]{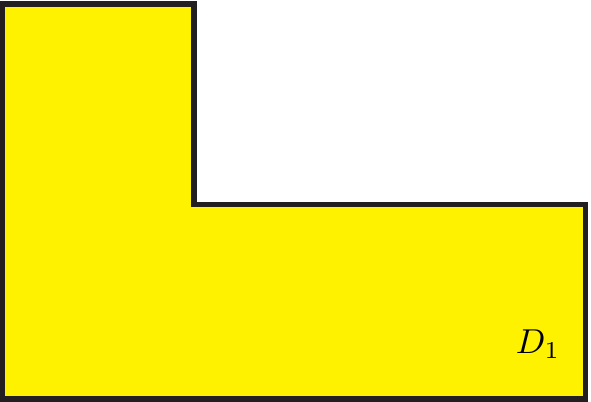}
\end{minipage}\hfill \begin{minipage}{0.5\textwidth}
\includegraphics[width=4.5cm]{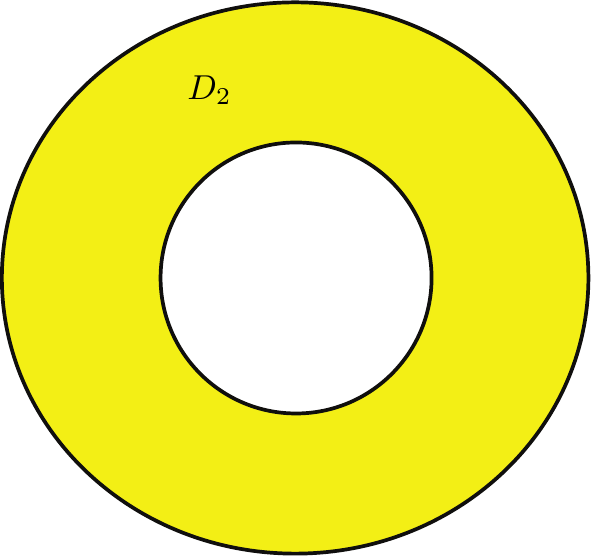}
\end{minipage}\\[0.2cm] 

\begin{minipage}{0.5\textwidth}
\includegraphics[width=5cm]{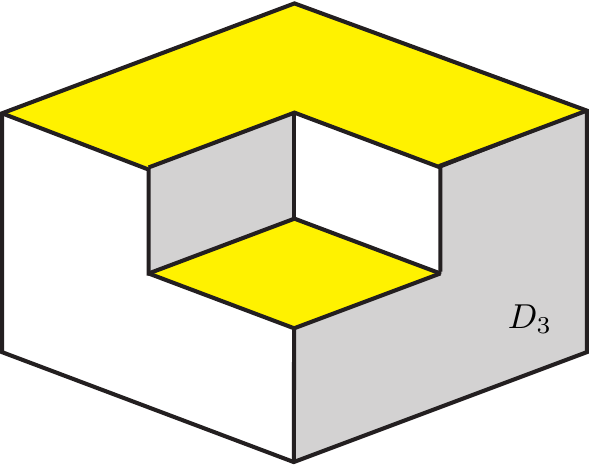}
\end{minipage}\hfill \begin{minipage}{0.5\textwidth}
\includegraphics[width=5cm]{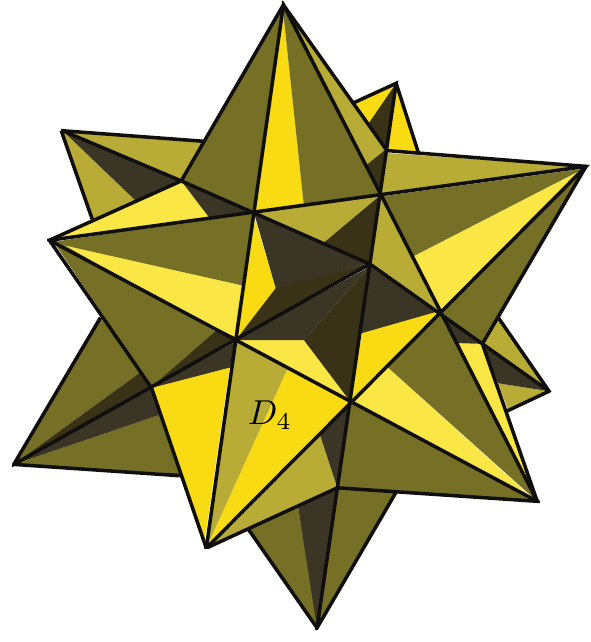}
\end{minipage}\\

Whereas the decomposition of the {icosahedron} into polyhedral cones and a smooth domain is obvious, 
we add a remark concerning $D_3$.
An open neighbourhood of the Fichera corner can be obtained as the image of a diffeomorphic map 
applied to a polyhedral cone (alternatively we could have defined polyhedral cones with an opening angle larger than $\pi$).
Finally, domains with slits and cusps as indicated below are  excluded by Definition \ref{standard} and not of polyhedral type.\\ 

\begin{minipage}{\textwidth}
\begin{center}\includegraphics[width=7cm]{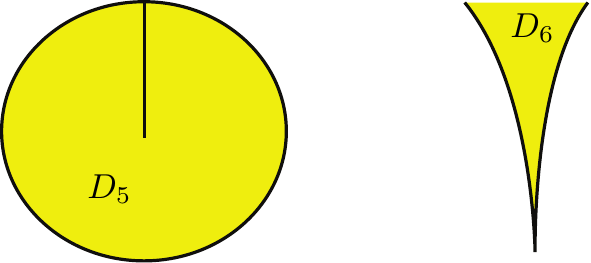}\end{center}
\end{minipage}\\

\end{examples}

\begin{rem}
	In the literature scales of weighted Sobolev spaces are also discussed on far more general domains.
	Exemplary, let us mention the work of Schrohe and Schulze \cite{SS1,SS2}. The authors discuss 
	pseudodifferential operators on manifolds with conical singularities, that is, topological spaces
	$X\times[0,\infty)/X\times\{0\}$ with $X$ being a smooth $n$-dimensional compact manifold. In this
	context, with the help of local coordinates $(x,t)$, spaces $H^{s,\gamma}$ can be defined to contain
	functions for which $t^{\frac{n}{2}-\gamma}(t\partial_t)^k\partial^\alpha_x u(x,t)\in L_2$ for all
	$k+|\alpha|\leq s\in\N$. In case of a smooth cone as in Case I (i.e. $X$ being a smooth submanifold
	of $\mathbb{S}^n$) the coordinate $t$ is equivalent to the weight function $\rho$, the (euclidean)
	distance to the	origin. Moreover, derivatives $\partial^\alpha_x u$ w.r.t. the spherical coordinates
	in $X$ can be expressed in terms of $\rho^{|\beta|}\partial^\beta u$ w.r.t. the standard cartesian
	coordinates. In other words, on smooth cones and for $s\in\N$ those spaces $H^{s,\gamma}$ correspond
	to Kondratiev spaces $\calk^s_{\gamma-\frac{n}{2},p}$. However, note that the general definition of
	the spaces $H^{s,\gamma}$ immediately allows for fractional smoothness parameters, i.e. $s\in\R$.
\end{rem}

We continue with a few well-known properties of Kondratiev spaces.

\begin{lem}\label{diffeo}
Let $D$ be a domain of polyhedral type with singularity set $S$.
The space  $\calk^m_{a,p}(D,S)$ is invariant under $C^\infty$ diffeomorphisms, i.e.,
if $\eta: D \to U$ denotes a $C^\infty$ diffeomorphism, then the function 
$u:~ D \to \R$ belongs to  $\calk^m_{a,p}(D,S)$ if, and only if,  
the function 
$u \circ \eta^{-1}:~ U \to \R$ belongs to  $\calk^m_{a,p}(U,\eta(S))$. Furthermore, 
$\|\, u\, |\calk^m_{a,p}(D,S)\|$ and $\|\, u\circ \eta^{-1}\, |\calk^m_{a,p}(U,\eta(S))\|$ are equivalent.
\end{lem}

\begin{proof}
 For convenience of the reader we give a proof. For unweighted Sobolev spaces such a result is well-known, we refer to 
Adams \cite[Thm.~3.35]{Ada}.
 \\
{\em Step 1.} For the time being  we restrict ourselves to the standard situations described in Case I - Case IV.
Recall that in these specialized situations we do not deal with the distance to the associated singularity set.
We need a common notation. Let $(R,N)$ refer to one of the above four cases.
We shall need a geometrical property of the underlying domain.
Concentrating on Case I - Case IV,  it is  obvious that there exists some $\varepsilon >0$ such that for all $x \in N$ and all 
$y \in B_\varepsilon (x) \cap R$ the lines connecting $x$ and $y$ are contained in $R$. 
Let $\eta = (\eta_1, \ldots , \eta_d)$. For all such pairs $x$ and $y$ it follows 
\[
 |\eta (x) - \eta (y) |\le \Big(\sup_{\xi \in D} \, \max_{j,i=1, \ldots\, d}\,  \Big|\frac{\partial}{\partial \xi_j} \eta_i (\xi)\Big|\Big) \, |x-y|\, .
\]
Of course, $\overline{R}$ and it's image $\overline{U}= \eta (\overline{R})$ are compact.
Hence, there must  exist a constant $C_\eta>0$ such that 
\be\label{ws-05}
 \frac{1}{C_\eta}\, |x-y| \le |\eta (x) - \eta (y) | \le C_\eta \, |x-y| \, , \qquad x,y \in \overline{R}\, .  
\ee
Let  $\tau:= \eta^{-1}$.
The Faa di Bruno formula for derivatives of the composition, cf.  \cite[Theorem 3.4]{CS}, gives us
\be\label{faadibruno1}
\partial^{\alpha} (u\circ \tau)(y) = \sum_{1 \le |\gamma| \le |\alpha|} (\partial^\gamma u )( \tau(y)) 
\, \sum_{\beta^1_1, \ldots \, , \beta^{\gamma_d}_d} \, 
c_{\alpha,\beta^1_1, \ldots \, , \beta^{\gamma_d}_d}\,  \prod_{j=1}^d\prod_{k=1}^{\gamma_j} \partial^{\beta_j^k} \tau_j (y)
\ee
where the second sum runs over all multiindices
$$
	\beta_1^1,\ldots,\beta_1^{\gamma_1},\ldots,\beta_d^1,\ldots,\beta_d^{\gamma_d}
		\in\N^d_0\setminus\{0\}
	\quad\mbox{satisfying}\quad
	\alpha = \sum_{j=1}^d\sum_{k=1}^{\gamma_j}\beta_j^k\,,
$$
with appropriate positive constants $c_{\alpha,\beta^1_1, \ldots \, , \beta^{\gamma_d}_d}$.
We put 
\beqq
\rho_R (y)& := & \min (1, \dist (y, N))\, , \qquad y \in R\, , 
\\
\rho_U (y) & := &  \min (1, \dist (y, \eta(N)))\, , \qquad y \in U\, .
\eeqq 
Let $x \in N$ be fixed.
Hence, the boundedness of the derivatives $\partial^{\beta_j^k}\tau_j $  and a change of coordinates  leads to 
\beq
\label{w-01}
\int_{\eta (B_\varepsilon (x) \cap R)} |\rho_U(y)^{|\alpha|-a} \partial^\alpha (u\circ \tau)(y)|^p\,dy & \lesssim & 
\int_{\eta (B_\varepsilon (x) \cap R)} \Big|\rho_U(y)^{|\alpha|-a}  \sum_{1 \le |\gamma|\le |\alpha|} (\partial^\gamma u)(\tau (y)) \Big|^p\,dy
\nonumber
\\
& \lesssim & \sum_{1 \le |\gamma|\le |\alpha|}
\int_{\eta (B_\varepsilon (x) \cap R)} \Big|\rho_U (y)^{|\alpha|-a}   (\partial^\gamma u)(\tau (y)) \Big|^p\,dy
\nonumber
\\
& \lesssim & \sum_{1 \le |\gamma|\le |\alpha|}
\int_{B_\varepsilon (x) \cap R} \Big|\rho_U (\eta (z))^{|\alpha|-a}   (\partial^\gamma u)(z) \Big|^p\,dz
\, .
\eeq
Applying \eqref{ws-05}  we can replace $\rho_U (\eta (z))$ by $\rho_R (z)$ itself on the right-hand side. 
We define
\[ 
N_{\varepsilon/2}: = \bigcup_{x\in N} \overline{B_{\varepsilon/2}(x)} \, .
\]
Obviously $N_{\varepsilon/2}$ is a compact set which has an open covering given by $\bigcup_{x\in N} {B_{\varepsilon}(x)}$.
The Theorem of Heine-Borel yields the existence of finitely many points $x_1, \ldots \, , x_J$ in $ N$
such that 
\[
 N_{\varepsilon/2} \subset \bigcup_{j=1}^J {B_{\varepsilon}(x_j)} \, .
\]
 Furthermore, let
\[
R_0 := \{x \in R: ~ \dist(x,N) \ge \varepsilon/2 \}\, .
\]
On this set $R_0$ the function $\rho_R$ is equivalent to $1$ and at the same time,
the function $\rho_U$ is equivalent to $1$ on $U_0:= \eta(R_0)$.
Hence,  on this part of $R$,
we may use the invariance with respect to the unweighted case, i.e.,
\be\label{w-02}
 \|\, u\circ \tau\, |\calk^m_{a,p}(U_0,\eta(N))\|\lesssim \| \, u\, |\calk^m_{a,p}(R_0,N)\|\, ,
\ee
see Adams \cite[Thm.~3.35]{Ada}.  Clearly, 
\[
 R \subset \Big(R_0 \cup \bigcup_{j=1}^J {B_{\varepsilon}(x_j)}\Big)\, .
\]
Finally, summing up the inequalities \eqref{w-01} with $x$ replaced by $x_j$ and  taking into account \eqref{w-02}, we get
\[
 \|\, u\circ \tau\, |\calk^m_{a,p}(U,\eta(N))\|\lesssim \| \, u\, |\calk^m_{a,p}(R,N)\|\, .
\]
Interchanging the roles of $\tau$ and $\eta$ we obtain the reverse inequality.
\\
{\em Step 2.} The necessary modifications for the general case are obvious.
\end{proof}

Now we are going to discuss the importance of the existence of an associated decomposition of unity.

\begin{lem}\label{deco}
 Let $D$, $(U_i)_i$, and $\Lambda_j$ with  $j=1, \ldots ,5$,  be as in Definition \ref{standard}. 
Moreover, denote by $S$ the singularity set of $D$ and let  $(\varphi_i)_i$ be a decomposition of 
unity subordinate to our covering, i.e.,
$\varphi_i \in C^\infty$, $\supp \varphi_i \subset U_i$, $0 \le \varphi_i \le 1$ and 
\[
 \sum_{i} \varphi_i (x) = 1\qquad \mbox{for all}\quad x \in \overline{D}.
\]
We put $u_i := u \, \cdot \, \varphi_i$ in $D$. \\
{\rm (i)} {If $u\in \calk^m_{a,p}(D,S)$ then }
\beqq
\|\,  u \, |\calk^m_{a,p}(D,S)\|^ * & : = & \max_{i\in \Lambda_1}\,  \|u_i|W^m_{p}(D \cap U_i)\| + 
  \max_{i\in \Lambda_2}\,  \|u_i (\eta_i^{-1} (\, \cdot\, ))\, | \calk^m_{a,p}(K,\{0\})\| 
\\
 & + &
\max_{i\in \Lambda_3}\,  \|\, u_i (\eta_i^{-1} (\, \cdot\, ))\, | \calk^m_{a,p}(P, \Gamma)\| 
+\max_{i\in \Lambda_4} \, \|\, u_i (\eta_i^{-1} (\, \cdot\, ))\, | \calk^m_{a,p}(I, M_\ell)\|
\\
 & + &
\max_{i\in \Lambda_5}\,  \|\, u_i (\eta_i^{-1} (\, \cdot\, ))\, | \calk^m_{a,p}(Q,M)\| 
\eeqq
generates an equivalent norm on $\calk^m_{a,p}(D,S)$.
\\
{\rm (ii)}
If $u:~D \to \mathbb{C}$ is a function such that the pieces $u_i$ satisfy
\begin{itemize}
 \item   $u_i \in W^m_{p}(D \cap U_i)$, $i\in \Lambda_1$;
  \item   $u_i (\eta_i^{-1} (\, \cdot\, )) \in \calk^m_{a,p}(K,\{0\})$, $i\in \Lambda_2$;
  \item   $u_i (\eta_i^{-1} (\, \cdot\, )) \in \calk^m_{a,p}(P, \Gamma)$, $i\in \Lambda_3$;
   \item   $u_i (\eta_i^{-1} (\, \cdot\, )) \in \calk^m_{a,p}(I, M_\ell)$, $i\in \Lambda_4$;
 \item   $u_i (\eta_i^{-1} (\, \cdot\, )) \in \calk^m_{a,p}(Q, M)$, $i\in \Lambda_5$;
  \end{itemize}
then $u \in  \calk^m_{a,p}(D, S)$ and 
\[
\|\,  u \, |\calk^m_{a,p}(D,S)\|\lesssim  \|\,  u \, |\calk^m_{a,p}(D,S)\|^ *\, .  
\]
\end{lem}

\begin{proof}
{\em Step 1.} Proof of (i).
Let $S_i := S\cap \overline{(U_i \cap D)}$. We claim 
\be\label{extra1}
\min (1, \dist (x,S)) \asymp  \min (1, \dist (x,S_i))\, , \qquad x \in   U_i \cap D\, .
\ee
To prove this we argue by contradiction. Let us assume that  there exists a real number $\varepsilon >0$ and a sequence 
$(x_j)_{j=1}^\infty \subset U_1 \cap D$ such that 
\[
 \dist (x_j,   S_1)  \ge \varepsilon 
\qquad \mbox{and}\qquad \lim_{j\to \infty} \dist (x_j,S)=0\, .
\]
Hence, there exists a subsequence $(x_{j_\ell})_\ell $ which is convergent with limit  $x_0$.
Necessarily $x_0 \in S$ and $x_0 \in \overline{U_1 \cap D}$. But this implies $x_0 \in S_1$, which  is a contradiction to our 
previously made assumption. The boundedness of $D$ yields the claim \eqref{extra1}.
Observe that this  implies 
\beq\label{ws-10}
\| \, u \, |\calk^m_{p,a}(D,S)\|  & = &  \Big(\sum_{|\alpha|\leq m} \int_D |\rho(x)^{|\alpha|-a} \sum_i \partial^\alpha u_i (x)|^p\,dx\Big)^{1/p}
\nonumber
\\
& \le & \sum_i \Big(\sum_{|\alpha|\leq m} \int_{U_i \cap D}  |\rho(x)^{|\alpha|-a}  \partial^\alpha u_i (x)|^p\,dx\Big)^{1/p}
\nonumber
\\
 & \lesssim & \sum_i \|\,  u_i\, |\calk^m_{p,a}(U_i \cap D,S_i)\|\, .
\eeq
We split the summation on the right-hand side into the five sums $\sum_{i\in \Lambda_j} $, $j=1, \ldots \, 5$.
In the first case we shall use
\[
 \dist (U_i \cap D,S) \geq c > 0\, , \qquad i \in \Lambda_1.
\]
This yields $c\leq \min_{i \in \Lambda_1} \dist (U_i \cap D,S)$ and consequently, by using $\rho \le 1$,  
\beq
\sum_{i\in \Lambda_1} \|\,  u_i\, |\calk^m_{p,a}(U_i \cap D,S_i)\| & \lesssim &  
\sum_{i\in \Lambda_1} \Big(\sum_{|\alpha|\leq m} \int_{U_i \cap D}  |\partial^\alpha u_i (x)|^p\,dx\Big)^{1/p}
\nonumber
\\
\label{ws-09}
&\lesssim & \sum_{i\in \Lambda_1}  \|u_i|W^m_{p}(D \cap U_i)\| \lesssim \max_{i\in \Lambda_1}  \|u_i|W^m_{p}(D \cap U_i)\|.
\eeq
Concerning the remaining terms we shall apply Lemma \ref{diffeo} and find
\beq
\label{ws-06}
\|\,  u_i\, |\calk^m_{p,a}(U_i \cap D,S_i)\|& \lesssim &
\|\,  u_i(\eta_i^{-1} (\, \cdot\, ))\, |\calk^m_{p,a}(K,{\{0\}})\|\, , \qquad i \in \Lambda_2\, , 
\\
\label{ws-07}
\|\,  u_i\, |\calk^m_{p,a}(U_i \cap D,S_i)\|& \lesssim &
\|\,  u_i(\eta_i^{-1} (\, \cdot\, ))\, |\calk^m_{p,a}(K,P)\|\, , \qquad i \in \Lambda_3\, ,
\\
\label{ws-08b}
\|\,  u_i\, |\calk^m_{p,a}(U_i \cap D,S_i)\|& \lesssim &
\|\,  u_i(\eta_i^{-1} (\, \cdot\, ))\, |\calk^m_{p,a}(I,M_l)\|\, , \qquad i \in \Lambda_4\, ,
\\
\label{ws-08c}
\|\,  u_i\, |\calk^m_{p,a}(U_i \cap D,S_i)\|& \lesssim &
\|\,  u_i(\eta_i^{-1} (\, \cdot\, ))\, |\calk^m_{p,a}(Q,M)\|\, , \qquad i \in \Lambda_5\, .
\eeq
Inserting \eqref{ws-09}-\eqref{ws-08c} into \eqref{ws-10} we have proved 
$\| \, u \, |\calk^m_{p,a}(D,S)\|\lesssim \| \, u \, |\calk^m_{p,a}(D,S)\|^*$.
In view of Lemma \ref{diffeo} the reverse inequality is obvious.
\\
{\em Step 2.} Proof of (ii).
Lemma \ref{diffeo} yields $u_i \in \calk^m_{p,a}(D,S)$ for all $i$. Hence 
$\sum_{i}u_i \in \calk^m_{p,a}(D,S)$. But $\sum_i u_i = u$.
\end{proof}

If we omit the usage of a resolution of unity, we can decompose a polyhedral domain without the explicit inclusion of polyhedral cones.
Thus let the polyhedral cone $Q$ and the set $M$ be as in Case IV,
with edges $M_1,\ldots,M_n$ and vertex in $0$. The angles in the plane $x_3=1$ are denoted by $\theta_1, \ldots , \theta_n$, 
see the figure in Case IV on page \pageref{def-norm-poly-cone}. 

In Case II we discussed the notion of a nonsmooth cone. In what follows  we need a simple modification.
Let $R>0$ denote the radius of the circle in $x_3 = 1$.
Then the nonsmooth cone is given by
\[
P=\{(x_1,x_2, x_3): ~ {x_1,x_2 > 0}, ~ 0 < x_3 < 1, ~ {x_1 ^2+ x_2 ^2} < R^2 \, x_3 ^2\}\, .
\]
Let $0 < \lambda < 4$.
For fixed $x_3$, by using polar coordinates $(r,\varphi)$, the mapping  
\[
 \eta: ~ (r,\varphi)\mapsto (r, \lambda \varphi)
\]
yields a diffeomorphism of $P$ onto the set  
\[
 \eta (P):= {P'}=\{{(r\cos \varphi, r\sin \varphi,x_3)}: 
~ 0 < x_3 < 1, ~ r^2 < R^2 \, x_3 ^2\, , ~ 0 < \varphi < \lambda \pi/2\}\, .
\]
In what follows,  these sets $\eta (P)$ as well as rotated versions will also be called nonsmooth cones.
\\
{~}\\
\begin{minipage}{0.5\textwidth}
Now we continue our construction.
For every edge $M_j$ with angle $\theta_j$ we choose
a nonsmooth cone $P_j$ {with axis $M_j$} as described above with sufficiently small opening angle $\gamma_j$, so
that none of  these cones intersect. 
Next we choose  a further  nonsmooth cone $\widetilde P_j$ with axis $M_j$ and opening angle
$\gamma_j/2$.
Then there exists a smooth cone $\widetilde K$ (i.e., a 
	cone whose intersection with $S^2$ is a smooth subset) such that
\[	
\Big(Q\setminus\bigcup_{j=1}^n P_j\Big) \subset \widetilde K
		\subset \Big(Q\setminus\bigcup_{j=1}^n\overline{\widetilde P_j}\Big) \, .
\]
Clearly, there exists a $C^\infty$-diffeomorphism $\sigma $, $\sigma (0)=0$,  mapping
	$\widetilde K$ onto a rotationally symmetric cone $K$ as in Case I.
\end{minipage}\hfill \begin{minipage}{0.5\textwidth}\quad
\includegraphics[width=7cm]{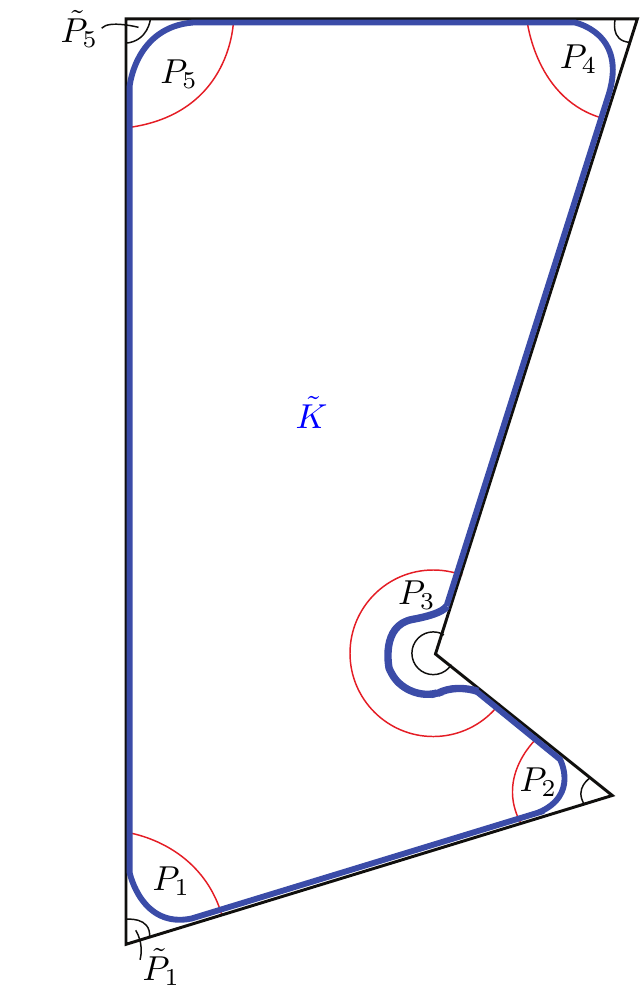}
\end{minipage}

This construction ensures that $\widetilde K$ has a non-empty intersection with each cone
$P_j$; on $P_j\subset Q$ the distance to $M$ is equivalent to the distance to
	$M_j$; and on $\widetilde K$ the distance to $M$ is equivalent to the distance to the
	vertex $0$. Altogether we have a decomposition 
\[
Q=\widetilde K\cup\bigcup_{j=1}^n P_j\, ,
\] 
where each point $x\in Q$ belongs to at most $2$ of the subsets.  Let $\rho(x):= \min (1,\dist (x,M))$. The latter then
immediately implies
\begin{align*}
\|u|\calk^m_{a,p}(Q,M)\|^p
& \sim \,  \sum_{|\alpha|\leq m}\Biggl(\int_{\widetilde K}
\Bigl|\rho(x)^{|\alpha|-a}\partial^\alpha u(x)\Bigr|^p dx +\sum_{j=1}^n\int_{P_j}
\Bigl|\rho(x)^{|\alpha|-a}\partial^\alpha u(x)\Bigr|^p dx\Biggr)
\\
&\sim  \, \|u|\calk^m_{a,p}(\tilde{K},\{0\})\|^p
+ \max_{j=1, \ldots \, n}\,  \|u|\calk^m_{a,p}(P_j,M_j)\|^p\, .
\end{align*}

Using an appropriate  diffeomorphic map we see that the above considerations lead to the following.

\begin{lem}\label{deco2}
Let $P$ be a polyhedral {cone} and let $M$ be the particular set  from Case IV. 
\\
{\rm (i)} Then 
\[
\|\,  u \, |\calk^m_{a,p}(P,M)\|^{**}
 :=   \| u \circ \sigma^{-1} \, | \calk^m_{a,p}(K,\{0\})\| 
+ \max_{j=1,\ldots \, , n} \, \|\, u\,| \calk^m_{a,p}(P_j, M_j)\| 			
\]
generates an equivalent norm on $\calk^m_{a,p}(P,M)$. 
\\
{\rm (ii)} Let $u_0: ~\widetilde K \to \mathbb{C}$ and $u_j:~P_j \to \mathbb{C}$, $j=1, \ldots \, , n$, 
be given functions satisfying $u_0(x)=u_j(x)$ for $x\in \widetilde K \cap P_j$ and
 $u_i(x)=u_j(x)$ for 	$x\in P_i \cap P_j$, respectively, and
\begin{itemize}
		\item   $u_0  \in \calk^m_{a,p}(\widetilde K,\{0\})$,
		\item   $u_j  \in \calk^m_{a,p}(P_j, M_j)$,
			$j=1 , \ldots \, , n$.
\end{itemize}
Then the function $u:P \to\mathbb{C}$,  piecewise-defined   via $u(x)=u_0(x)$, $x\in \widetilde K$,
and  $u(x)=u_j(x)$, $x\in P_j$, $j=1, \ldots \, n$,
satisfies $u \in  \calk^m_{a,p}(P, M)$ and 
\[
\|\,  u \, |\calk^m_{a,p}(P,M)\|\lesssim  \|\,  u \, |\calk^m_{a,p}(P,M)\|^{**}\, . 
\]
\end{lem}

\begin{proof}
It will be enough to  comment on (ii). 
The compatibility condition implies that the function 
$u$ is well-defined. 
Moreover, since the sets $\widetilde{K}$, $P_j$ are assumed to be open (hence their
intersections are open as well), also the weak differentiability of the pieces $u_j$
carries over.
\end{proof}

\begin{example}\label{ex-dobrowolski}
There is one more interesting example {of a polyhedral domain} 
we wish to mention. It is very much in the spirit of the decomposition of the polyhedral cone itself and  motivates a 
further moderate  generalization.

\centerline{ \includegraphics[width=8cm]{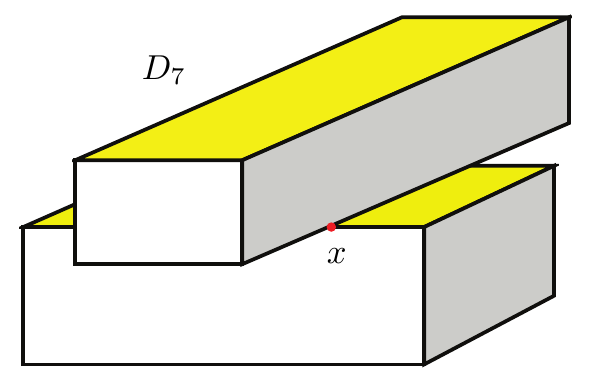} }

{~}\\
Example   $D_7$ is taken from  \cite[Ex.~6.5]{dob10}  and consists of two cuboids lying on top of each other.  
Furthermore, it is a domain of polyhedral type which is not Lipschitz 
(at the indicated point $x$ it is not possible to turn the polyhedron such that the inner part of the domain lies behind the boundary). For the same reason, this domain also does not possess the segment property. However, it is easily seen that its subdomains
$D_7^\delta=\{x\in D_7:\dist(x,S)>\delta\}$ have this property, where the singularity set $S$ as usual consists of all edges. Let us stress the fact that this is a first example where the properties of the spaces $\calk^m_{a,p}(D_7,S)$ are determined not by properties of the domain $D_7$, but of suitable subdomains ($D_7^\delta$). Similar behaviour is rather typical for Kondratiev spaces, as can be seen later on from Proposition \ref{zerlegungdereins} and some of its implications, e.g. in Propositions \ref{hilfe} and
\ref{lemma-embedding-sobolev}.

A simple decomposition in the spirit of Definition \ref{standard} can be given as  follows.
The only critical point is $x$ itself. Hence, we only deal with an open neighbourhood $U$ of $x$.
The set $U\cap D$ can be split into the union of four nonsmooth cones (along the edges) $P_1, \ldots , P_4$ with common vertex in $x$.
The remainder of $U\cap D$ is covered by  the diffeomorphic images $\tilde{K}_1,\tilde{K}_2$  of two smooth 
cones $K_1,K_2$ with vertex in 
$x$. Here, as in the case of the polyhedral cone, we may suppose that the intersections  
$\tilde{K}_1 \cap P_1$, $\tilde{K}_1 \cap P_2$, $\tilde{K}_2 \cap P_3$, $\tilde{K}_2 \cap P_4$, are open and nontrivial,
\[
\tilde{K}_1 \cap P_3 = \tilde{K}_1 \cap P_4 = \tilde{K}_2 \cap P_1 = \tilde{K}_2 \cap P_2 = \emptyset\, ,
\]
 and that $\overline{P_j} \cap \overline{P_\ell} =\{x\}$, $j \neq \ell$. 
This immediately yields the counterpart of Lemma \ref{deco2} for $D_7$. Note, however, that none of those cones contains $\{x\}$, so this does {\it not} yield a cover for $\overline{D_7}$.
\end{example}

The moderate generalization we have in mind is as follows.

\begin{defi}\label{gd}
Let $D\subset \Rd$ be a  domain and let $M$ denote a closed subset of its boundary $\partial D$.
The pair $(D,M)$ is of generalized polyhedral type
if either 
\begin{itemize}
 \item $(D,M)$ coincides with one of the pairs in Cases I-IV. 
\item $(D,M)$ is of polyhedral type. 
\item The set $D$ {with nontrivial closed subset  $M$ of $\partial D$}  satisfies the following conditions:
\begin{itemize}
\item[(a)] There exist finitely many  open sets $D_1, \ldots \, , D_N$  such that
\[
{D} = \bigcup_{j=1}^N {D}_j \, .
\]
\item[(b)] Denote $M_j := M \cap \overline{D}_j$, $j=1, \ldots \, , N$. Then each pair $(D_j,M_j)$ is the diffeomorphic image
of one of the Cases I-IV.
\item[(c)] Each $x \in D$ should belong to at most two different sets $D_j$ and  
\[
|\{\ell \in \{1, \ldots \, ,N\}\setminus\{j\}:
	~ D_j \cap D_\ell \neq \emptyset \}|\ge 1\, , \qquad j=1, \ldots \, , N\, . 
\]
\end{itemize}
\end{itemize}
\end{defi}

\begin{rem}
	Let us give some comments on this definition. At  first glance it may not be clear that indeed this
	definition is more general than the previous one (Definition \ref{standard}). The difference between
	the definitions becomes clearer when  considering a double cone (or more generally, a cone
	$K$ where the intersection with the unit sphere $S^d$ is a disconnected open subset of $S^d$). While
	naively such a domain clearly can be decomposed into connected cones, this decomposition does NOT
	yield an admissible cover in the sense of Definition \ref{standard} (note that there we require
	$\overline{D}$ to be covered). What is more, such a cover does not exist: on any open set covering the
	common vertex the distance to the singular set can neither be equivalent to the distance to one
	fixed edge, nor can it be a polyhedral cone. In other words, a domain consisting of two polyhedral
	cones whose closures intersect in $\{0\}$ is not a domain of polyhedral type in the sense of
	Definition \ref{standard}.
	
	Admittedly, such a domain is not of generalized polyhedral type either -- condition
	(c) is violated -- and disconnected domains might be of minor interest at best, but it helps to
	clarify the matter.
	
	It is exactly this situation we  come across when considering Dobrowolski's example $D_7$: a
	neighbourhood of the critical point $x$ can be decomposed into several smooth and nonsmooth or
	polyhedral cones as described above in Example \ref{ex-dobrowolski}, with common vertex $x$, but $x$ itself cannot be covered by any
	open neighbourhood as required by Definition \ref{standard}. On the other hand, the described
	decomposition clearly shows $D_7$ to be of generalized polyhedral type.
\end{rem}

\begin{lem}\label{deco3}
Let the pair $(D,M)$ be of generalized polyhedral type.
Then the statements from Lemma \ref{deco2}, appropriately modified, remain true. 
\end{lem}

Note, however, that there is no counterpart of Lemma \ref{deco} (see also the discussion following Remark \ref{rem-def-domain}).

\begin{rem}
 \rm 
 All domains of generalized  polyhedral type satisfy the cone condition, cf. \cite[Def.~4.6]{AF03}. But, in general, 
they do not have a Lipschitz boundary, see 
the example $D_7$. For our investigations within this article we do not require any additional regularity assumptions on the domain (or its boundary) beyond being of (generalized) polyhedral type.
\end{rem}


\subsection{Extensions}


Stein's linear extension operator $\mathfrak{E}$, see \cite[VI.3.2]{St70}, has become a standard tool in the 
framework of Sobolev spaces on Lipschitz domains. It can be used in the framework of Kondratiev spaces as well.
The following Proposition represents a particular case of a more general result 
which can be found in \cite{Hansen}.
We need one more notation, compare with \eqref{ws-15}:
\be\label{ws-16}
\R^\ell_* := \{x \in \Rd:~x_1 = \ldots = x_{d-\ell }=0\}. 
\ee
Clearly, if $\ell=1$ we shall simply write $\R_*$.
For brevity we also put $\R^0_* := \{0\}$ and 
$\calk^m_{a,p}(\Rd, \R^\ell_*):= \calk^m_{a,p}(\Rd \setminus \R^\ell_*, \R^\ell_*)$, $0 \le \ell < d$.

\begin{prop}\label{extension}
Let $d\ge 2$, $1 \le p <\infty$, $a\in \R$, and $m \in \N$.
\\
{\rm (i)} Let $K$ be our smooth cone from {\bf Case I}.
Then the Stein extension operator $\mathfrak{E}$ yields a linear and bounded mapping of 
$\calk^m_{a,p}(K,\{0\})$ into $\calk^m_{a,p}(\Rd,\{0\})$.
\\
{\rm (ii)} Let $P$ and $\Gamma$ be as in  {\bf Case II}.
Then the Stein extension operator $\mathfrak{E}$ yields a linear and bounded mapping of 
$\calk^m_{a,p}(P,\Gamma)$ into $\calk^m_{a,p}(\Rd,\R_*)$.
\\
{\rm (iii)} Let $I$ and $M_\ell$, $1\le \ell < d, $ be as in  {\bf Case III}.
Then the Stein extension operator $\mathfrak{E}$ yields a linear and bounded mapping of 
$\calk^m_{a,p}(I,M_\ell)$ into $\calk^m_{a,p}(\Rd,\R^\ell_*)$.
\end{prop}

With Proposition \ref{extension} and Lemma \ref{deco} at hand we can reduce 
basic properties of the spaces $\calk^m_{a,p}(D,S)$ (under some extra conditions on $D$) to the  model cases  
$\calk^m_{a,p}(\Rd, \R^\ell_*)$, $0 \le \ell < d$. In this model setting we find for the weight function
\begin{equation}\label{eq-weight-model}
	\rho(x_1,\ldots,x_{d-\ell},\ldots,x_d)
		=\min\biggl(1,\Bigl(\sum_{i=1}^{d-\ell}|x_i|^2\Bigr)^{1/2}\biggr)\,.
\end{equation}

\begin{rem}
{\rm 
For a moment we return to a discussion of Lemma \ref{deco} and Lemma \ref{deco2}.
By Proposition \ref{extension}  we have extension operators $\mathfrak{E}_0, \ldots\, , \mathfrak{E}_n$ 
for these different types of Kondratiev spaces
appearing in Lemma \ref{deco2}(ii). 
We define
\[
u:= \sum_{j=0}^n \mathfrak{E}_j u_j \, .
\]
Then $u_{|_{\widetilde{K}}} \in \calk^m_{a,p}(\widetilde K,\{0\})$ and 
$u_{|_{P_j}} \in \calk^m_{a,p}(P_j, M_j)$, $j=1, \ldots \, , n$ follows. However, in general we do not have coincidence of 
$u_{|_{\widetilde{K}}}$ with $u_0$ and of $u_{|_{P_j}}$ with $u_j$.
}
\end{rem}


\subsection{A localization principle}

\label{local1}

The following decomposition of the norm of weighted Sobolev spaces  is in some sense standard.
We will allow a slightly greater generality than before.

Let $\Omega\subset \Rd$ be an open, nontrivial connected set and let $M$ be a closed nontrivial subset of the boundary.
Then we define 
\begin{equation}\label{Dj-set}
\Omega_j : = \{x\in \Rd: ~~2^{-j-1}<\rho(x)<2^{-j+1}\},\qquad j\in\N_0\, , 
\end{equation}
where 
$\rho (x):= \min (1, \dist (x,M))$, $x\in \Rd$.
Next we choose the largest number  $j_0\in \N_0$ such that  

\[
 \{x\in \Omega: ~ \rho (x) \ge  2^{-j_0+1}\}  = \emptyset \, .
\]
This implies
\[
\Omega \cap \Omega_{j_0}=  \{x\in \Omega: ~ 2^{-j_0 -1} < \rho (x) <  2^{-j_0+1}\}  \neq  \emptyset \, .
\]
Because $\Omega$ is open and connected, the 
continuity of $\rho$ yields

\be\label{w-03}
|\Omega \cap \Omega_j| >0 \qquad \mbox{for all} \quad j\ge j_0\, . 
\ee
Hence
\[
 \Omega =  \bigcup_{j=j_0}^\infty  (\Omega \cap \Omega_j)  \, .
\]
For technical reasons we need to distinguish the following two cases:
\\
a) $ \{x\in \Omega: ~~2^{-j_0} + 2^{-j_0-2} <\rho(x)<2^{-j_0+2} - 2^{-j_0 -2}\} = \emptyset$.
Then we define $j_1 := j_0$.
\\
b) $ \{x\in \Omega: ~~2^{-j_0} + 2^{-j_0-2} <\rho(x)<2^{-j_0+2} - 2^{-j_0 -2}\} \neq \emptyset$.
Then we define $j_1 := j_0-1$.
\\ 
Of course, we will need some regularity of $\Omega$.
We will use a condition guaranteeing that for $x \in \Omega$ an essential part of the  ball centred at $x$ and with radius proportional 
to the {distance of $x$ to $M$} lies inside $\Omega$.
We put  $\sigma := 2^{-j_0}$.

\begin{prop}\label{zerlegungdereins}
Let $1 \le p < \infty$, $a\in \R$, and $m \in \N$.
Let $\Omega, M, \rho$, and $j_1$  be as above. 
We assume that there exist two constants $c>0$ and $t\in (0,1)$ such that
\begin{itemize}
 \item for all $x \in \Omega$, $\dist (x,M)< \sigma$, the balls $B_{\lambda} (x)$ satisfy
\be\label{ball condition}
|B_{\lambda} (x) \cap \Omega|\ge c\, \lambda^{d} \qquad \mbox{for all} \quad \lambda \in \left[\frac t{32}\,  \rho(x), t\, \rho(x)\right]\, . 
\ee
\item for all $x \in \Omega$, $\dist (x,M) \ge  \sigma$,  the balls $B_{\lambda} (x)$ satisfy
\be\label{ball condition2}
|B_{\lambda} (x) \cap \Omega|\ge c\, \lambda^{d}\qquad \mbox{for all}\quad \lambda\le t\sigma . 
\ee
\end{itemize}
Then  there exist positive constants $A,B$, and a smooth  decomposition of unity 
$(\varphi_j)_{j\ge j_1}$ such that 
\begin{itemize}
\item $\varphi_j \in C^\infty (\R^d)$, $j \ge j_1$;
\item $\supp \varphi_j \subset \Omega_j$, $j \ge j_1$;
\item $0 \le \varphi_j(x) \le 1$ for all $j \ge j_1$ and all $x\in \Omega$; 
 \end{itemize}
\[
\sum_{j=j_1}^\infty\varphi_j(x)= 1  \qquad \mbox{for all} \qquad 	x\in \Omega\, , 
\]
and 
\be\label{ws-160}
 A \, \|u|\calk^m_{a,p}(\Omega,M)\|^p \le  \sum_{j=j_1}^\infty\|\varphi_j u|\calk^m_{a,p}(\Omega, M)\|^p
\le B \,  \|u|\calk^m_{a,p}(\Omega, M)\|^p
\ee
for all $ u\in\calk^m_{a,p}(\Omega, M)$.
\end{prop}

\begin{proof}
{\em Step 1.} Construction of the $(\varphi_j)_{j=j_1}^\infty$.
We define $\varepsilon_j := 2^{-j-4}$, $j \ge j_1$.
Associated are the  two sequences $(a_j)_j$ and $(b_j)_j$ given by
\[
 0 < a_j := 2^{-j-1} + \varepsilon_j < b_j := 2^{-j+1} - \varepsilon_j\, , \qquad j \ge j_1.
\]
Observe that both sequences are strictly monotone decreasing.
Then we put
 \[
\Omega_j^\varepsilon : =\{x\in \Omega: ~~a_j < \rho(x) < b_j \}, \qquad j\ge j_1.
\]
It follows 
\be\label{voll}
\Omega = \bigcup_{j=j_1}^\infty  \Omega_j^\varepsilon  \, .
\ee
Later on we need a further modification.
Define
\[
\widetilde{\Omega}_j^\varepsilon:= \{x \in \Omega:~\quad  2^{-j-1} + 2\varepsilon_j < \rho (x) < 2^{-j+1} - 2\varepsilon_j \} \, , 
\qquad j \in \N_0\, .
\]
Then we still have 
\be\label{vollb}
\Omega = \bigcup_{j=j_1}^\infty  \widetilde{\Omega}_j^\varepsilon  \, .
\ee
{{Moreover, as in \eqref{w-03} }}
we conclude  $\widetilde{\Omega}_j^\varepsilon \neq \emptyset$ if  $j \ge j_1$
(and therefore ${\Omega}_j^\varepsilon \neq \emptyset$, $j \ge j_1$, as well).
Here we need the particular definition of $j_1$. 
By $\cx_j $ we denote the characteristic function of $\Omega_j^\varepsilon$, $j \ge j_1$.
We claim that
\[
1 \le \sum_{j=j_1}^\infty \cx_j (x)\le 2
\]
holds for all $x \in \Omega$.
The lower bound is a trivial consequence of \eqref{voll}.
Since 
\[
\Omega_j^\varepsilon \cap  \Omega_{j+\ell}^\varepsilon = \emptyset \, , \qquad \ell \ge 2\, , 
\]
for all $j\in \N_0$, also the upper bound is immediate.
\\
By means of a standard mollification we turn these functions into 
smooth functions.
We define
\[
\omega (x):= \alpha \, \left\{  \begin{array}{lll}
e^{\frac{1}{1-|x|^2}} & \qquad & \mbox{if}\quad |x|<1\, , 
\\
0 && \mbox{otherwise},                                  
\end{array}\right.  
\]
where the positive constant $\alpha $ is chosen in such a way that $\int \omega (y)dy = 1$ holds.
Clearly,  $\omega \in C_0^\infty (\Rd)$, $\omega \ge 0$, $\supp \omega \subset \overline{B_{1} (0)}$ and 
\be\label{moll}
\inf \Big\{ \int_E \omega (y)\, dy:\quad E \subset \overline{B_{1} (0)}\, , ~~ |E|\ge c\} >0
\ee
for any fixed $c \in (0, |B_1(0)|]$.
We put  
\be\label{abl}
 \phi_j (x):= (t\, 2^{-j-4})^{-d} \int \omega \Big(\frac{x-y}{t\, 2^{-j-4}}\Big)\, \cx_j (y)\, dy\, , \qquad x \in \Rd\, , \quad j \ge j_1\, .
\ee
Clearly, $\phi_j \in C^\infty (\Rd)$. 
Concerning the supports it is easily seen that 
\[
 \supp \phi_j \subset \{y\in \Rd: \: \dist (y, \Omega_j^\varepsilon) < \varepsilon_j \} \subset {\Omega}_j \, .
\]
On the other hand, for $x \in \widetilde{\Omega}_j^\varepsilon$, $j \ge j_1$,  we find
\[
\phi_j (x)\ge  (t\, 2^{-j-4})^{-d} \int_{B_{t \, 2^{-j-4}} (x) \cap \Omega_j^{\varepsilon}} 
\omega \Big(\frac{x-y}{t\, 2^{-j-4}}\Big)\, dy\,.
\]
Because of   $5 \cdot  2^{-j-3} < \rho (x) < 15\cdot  2^{-j-3}$, $x \in \widetilde{\Omega}_j^\varepsilon$, we conclude
\[
 \Big[\frac t{32} \, 15\cdot 2^{-j-3}, ~t \, 5\cdot  2^{-j-3}\Big] {\subset} \Big[\frac t{32} \, \rho (x), t \, \rho(x)\Big]
\]
for all $x \in \widetilde{\Omega}_j^\varepsilon$.
Our restriction \eqref{ball condition} yields
\[
 |B_{t \, 2^{-j-4}} (x) \cap \Omega| \ge c \, (t \, 2^{-j-4})^d\, ,  \qquad x \in \widetilde{\Omega}_j^\varepsilon\, ,
\]
 where  $c>0$ is  independent of $j$ and $x$. Observe,  by construction
\[
B_{t \, 2^{-j-4}} (x) \cap \Omega = B_{t \, 2^{-j-4}} (x) \cap \Omega_j^\varepsilon \, , \qquad x \in \widetilde{\Omega}_j^\varepsilon\, .
\]
Now we continue our estimate of $\phi_j$ and obtain in view of \eqref{moll}
\beqq
\phi_j (x)
 & \ge &  (t\, 2^{-j-4})^{-d} \int_{B_{t \, 2^{-j-4}} (x) \cap \Omega_j^{\varepsilon}} 
\omega \Big(\frac{x-y}{t\, 2^{-j-4}}\Big)\, dy
\\
& = & \int_{E} 
\omega (z)\, dz \ge C>0
\eeqq
with $C$ independent of $x$ and $j$. Here $E$ is the image of 
$B_{t \, 2^{-j-4}} (x) \cap \Omega_j^{\varepsilon}$ under the transformation 
$z= (t\, 2^{-j-4})^{-1}\, (x-y)$.
The assumption \eqref{ball condition} yields a uniform bound from below for $|E|$ independent of $x$.
Clearly, $\sup_{x}|\phi_j (x)|\le 1$ and on $\Omega_j$ only $\phi_{j+\ell}$, $|\ell|\le 1,$ are not identically zero.
Taking into account \eqref{vollb}  we have proved
\be\label{uniform}
0 < C \le  \sum_{j=j_1}^\infty \phi_j (x) \le 2\, , \qquad x \in \Omega\, .
\ee
This allows us to proceed in the standard way: we put
\beqq
 \varphi_j (x) :=  \frac{\phi_j (x)}{\sum_{j=j_1}^\infty \phi_j (x)}\, , \qquad x \in \Omega, \quad  j \ge j_1\, .
\eeqq
 By construction these functions have all the properties as claimed.
\\
{\em Step 2.}  Because of $\supp \varphi_j \subset {\Omega}_j$ and on $\Omega_j$ only  $\varphi_{j+\ell}$, $|\ell|\le 1,$ are not identically zero,
we conclude
\begin{eqnarray}\label{ws-04}
\|\, u\, |\calk^{m}_{a,p} (\Omega,M)\|^p
& \lesssim & \sum_{|\alpha| \le m} \Big(\sum_{j=j_1}^\infty
\int_{{ \Omega \,  \cap \, \supp \varphi_j}} |\rho(x)^{|\alpha|-a} \, \partial^\alpha u (x)|^p\, dx
\nonumber
\\
& \lesssim  & \sum_{|\alpha| \le m} \Big(\sum_{j=j_1}^\infty 2^{-j(|\alpha|-a)p}
\int_{\Omega_j \cap \Omega} |\, \partial^\alpha (u \varphi_j)(x)|^p\, dx
\nonumber
\, .
\end{eqnarray}
Observe that 
\[
| \partial^\alpha \varphi_j (x)|  \lesssim \varepsilon_j^{-| \alpha |} \lesssim 2^{j | \alpha |}\, , 
\]
see \eqref{abl}. Applying this estimate with respect to the subsequence $(\varphi_{2j})_{j \ge j_1/2}$,  we find
\beq
 \sum_{|\alpha| \le m} \sum_{j \ge j_1/2} 2^{-2j(|\alpha|-a)p} && \hspace{-0.7cm}
\int_{\Omega_{2j} \cap \Omega} |\, \partial^\alpha (u \varphi_{2j})(x)|^p\, dx
\nonumber
\\
& \lesssim &  \sum_{|\alpha| \le m} \sum_{j\ge j_1/2} 
\int_{\supp \varphi_{2j}  \cap \Omega} |\,\rho(x)^{|\alpha|-a} \,  \partial^\alpha (u \varphi_{2j})(x)|^p\, dx
\nonumber
\\
& \lesssim &  \sum_{|\alpha| \le m} \sum_{\beta \le \alpha}
\sum_{j \ge j_1/2} \int_{\supp \varphi_{2j} \cap \Omega} |\, \,\rho(x)^{|\alpha|-a} \, 
\partial^\beta u (x)\, \partial^{\alpha -\beta} \varphi_{2j}(x)|^p\, dx
\nonumber
\\
& \lesssim &   \sum_{|\beta| \le m}
\sum_{j\ge j_1/2} \int_{\supp \varphi_{2j}\cap \Omega} |\, \,\rho(x)^{|\beta|-a} \, 
\partial^\beta u (x)\, |^p\, dx
\nonumber
\\
& \lesssim &   \sum_{|\beta| \le m}
 \int_{\Omega} |\, \,\rho(x)^{|\beta|-a} \, 
\partial^\beta u (x)\, |^p\, dx = \|\, u\, |\calk^{m}_{a,p} (\Omega, M)\|^p
\, ,
\eeq
where we used in the last line $\supp \varphi_{2j} \cap \supp \varphi_{2j+2} = \emptyset$.
Taking  a similar estimate with $\varphi_{2j+1}$ instead of  $\varphi_{2j}$  into account we obtain
\beq\label{ws-20}
&& \hspace*{-0.7cm}
\|\, u\, |\calk^{m}_{a,p} (\Omega,M)\|^p 
\nonumber
\\
& \sim & \sum_{|\alpha| \le m} \sum_{j=j_1}^\infty 2^{-j(|\alpha|-a)p}
\int_{\Omega_j \cap \Omega} |\, \partial^\alpha (u \varphi_j)(x)|^p\, dx 
\nonumber
\\
&\sim &  \sum_{j=j_1}^\infty \|\varphi_j u| \calk^m_{a,p}(\Omega, M)\|^p
\eeq
as claimed.
\end{proof}


\subsection*{Examples and comments}


\begin{lem}
 \label{zerlegungdereinsBeispiel1}
Let $1 \le p < \infty$, $a\in \R$, and $m \in \N$. Let $0 \le \ell <d$.
Then $\Rd \setminus \R^\ell_*$ satisfies the restrictions \eqref{ball condition} and \eqref{ball condition2} with respect to the set $M:= \R_*^\ell$. 
The decomposition of unity, constructed in the proof of Proposition \ref{zerlegungdereins}, 
has the following additional properties: 
 
\begin{itemize}
 \item $j_0 =0$, i.e., $\sum_{j=0}^\infty\varphi_j(x)= 1$   for all $x \in \Rd \setminus \R_*^\ell$;
 \item $\varphi_{0}(x) + \varphi_{1}(x) = 1 $ for all $x\in  \Omega_0$;
\item $\varphi_{j-1}(x) + \varphi_{j}(x) + \varphi_{j+1}(x) = 1 $ for all $x\in  \Omega_j$, $j \in \N$;
 \item $\varphi_j(x) = \varphi_1 (2^{j-1}x)$, $x \in \Rd$, $j \in \N$. 
 \end{itemize}
\end{lem}

\begin{proof}
Almost all properties of the decomposition of unity are immediate except for probably the last one.
 The sets $\Omega_j$ with respect to the pair $(\Rd\setminus \R^\ell_*, \R_*^\ell)$ have a very simple geometric structure.
 The transformation $J:~x \mapsto 2^{-j+1}x$, restricted to $\Omega_1 $, is a bijection  onto $\Omega_j$, $j \ge 1$.
 This is enough to show $\varphi_j(x) = \varphi_1 (2^{j-1}x)$, $x \in \Rd$, $j \in \N$.
\end{proof}

Mutadis mutandis one can prove also the following.

\begin{lem}\label{zerlegungderzwei}
Proposition \ref{zerlegungdereins} is applicable with respect to 
the smooth cone, see Case I, with respect to the specific nonsmooth cone, see Case II,  the specific 
dihedral domain, see Case III, and the polyhedral cone, see Case IV, always equipped with the appropriate sets $M$.
\end{lem}

\begin{rem}\label{tool}
 \rm
(i) Those localized characterizations of Kondratiev spaces 
can be found also in Maz'ya, Rossmann \cite[Lemmas 1.2.1, 2.1.4]{MR}
for smooth cones and dihedral domains.
\\
(ii) In  \cite[3.2.3]{T78} Triebel discusses function spaces defined by localized norms as in \eqref{ws-160}.
But he is working with  $M:= \partial \Omega$.
\end{rem}

The following Lemma shows that  a set $\Omega$ need not be a Lipschitz domain in order to satisfy the assumptions in 
Proposition \ref{zerlegungdereins}.
Moreover, we provide examples to show that a pair $(\Omega, M_1)$ may satisfy the restrictions in Proposition \ref{zerlegungdereins}, whereas for a  second pair
$(\Omega, M_2)$ this need not to be true.

\begin{lem}
{\rm (i)} The example {$D_7$}  of a domain of generalized polyhedral type in $\R^3$ on page
\pageref{ex-dobrowolski}  equipped with its singular set
(edges and vertices) satisfies the restrictions in Proposition \ref{zerlegungdereins}. 
\\
{\rm (ii)} Let 
${D_8} := \{(x_1,x_2): ~ -1 < x_1 < 1, ~ \sqrt{|x_1|} < x_2< 1 \}$.
We put $M_1:= \{(0,0)\}$. Then 
the pair $(D_8,M_1)$ does not fulfil \eqref{ball condition}.
\\
{\rm (iii)} Let 
 $M_2:= \{(x_1,x_2):~-1 < x_1 < 1, ~ \sqrt{|x_1|}=x_2 \}$. Then 
the pair $(D_8,M_2)$ satisfies \eqref{ball condition} and \eqref{ball condition2}.
\end{lem}

\subsection{A general strategy}
\label{general}

Beginning with  the next subsection we shall employ  
Proposition \ref{zerlegungdereins} and its consequences, see Lemmas 
 \ref{zerlegungdereinsBeispiel1} and  \ref{zerlegungderzwei}, always in the following way:
 \begin{itemize}
  \item First step: localization of the underlying domain $D$ by means of Proposition \ref{zerlegungdereins}.
  \item Second step: reduction to some standard situation (unweighted Sobolev spaces defined on $\Omega_0$ or $\Omega_1$)
  by using homogeneity arguments.
  \item Third step: application of some well-known properties of $W^m_p (\Omega_0)$, $W^m_p (\Omega_1)$.
  \item Fourth step: rescaling and a second application of Proposition \ref{zerlegungdereins}. 
 \end{itemize}


\subsection{A further  equivalent norm}
\label{sect-equiv-norm}


Given  $x\in \Omega$, let $R(x)$ consist of all points $y$ in $\Omega$ such 
that the line segment joining $x$ to $y$ lies entirely in $\Omega$. Put 
\[
\Gamma(x) := \{y\in R(x): ~|y-x|<1\},
\]
and let $|\Gamma(x)|$ denote the Lebesgue measure of $\Gamma(x)$. Then $\Omega$ satisfies the {\em weak cone condition} if there exists a number $\delta>0$ 
such that 
\[
|\Gamma(x)|\geq \delta\quad \text{for all }x\in \Omega. 
\]
For classical Sobolev spaces it is known that 
\begin{equation}\label{eq-norm-sob}
\|u|W^m_p(\Omega)\|_{\ast}:= \|u|L_p(\Omega)\| + \sum_{|\alpha|=m}\|D^{\alpha}u|L_p(\Omega)\|
\end{equation}
is an equivalent norm in $W^m_p(\Omega)$ as long as the underlying domain $\Omega$ satisfies the weak cone condition, cf. \cite[Thm. 5.2]{AF03}. 

We will show that also in the case of Kondratiev spaces $\mathcal{K}^m_{a,p}(D,S)$ defined on domains of polyhedral type, 
it suffices to consider the extremal 
derivatives ($|\alpha|=m$ and $|\alpha|=0$) to obtain an  equivalent norm. 

As a preparation we shall deal with the model case $\Omega:=\real^d\setminus \real^l_{\ast}$ with   $M:=\real^l_{\ast}$.

\begin{prop}\label{hilfe}
Let $1\leq p<\infty$, $m\in \N$,  $a\in \R$, and $0\le \ell < d$. Then 
\[
\|u|\mathcal{K}^m_{a,p}(\Rd,\R^\ell_\ast)\|_{\ast}
 :=   \|\rho^{-a} \, u| L_{p}(\Rd)\| +
 \sum_{|\alpha|=m}\|\rho^{m-a} \, \partial^{\alpha}u| L_p(\Rd)\| 
\]
is an equivalent norm in $\mathcal{K}^m_{a,p}(\Rd,\R^\ell_\ast)$. 
\end{prop}

\begin{proof}
Recall that in this model case we have the explicit expression \eqref{eq-weight-model} for the weight function $\rho$.
We shall employ the partition of unity $\{\varphi_j\}_{j\in \nat_0}$ from Proposition \ref{zerlegungdereins} with 
\[
\supp \varphi_j\subset \Omega_j:=\{x\in \real^d: \ 2^{-j-1}<\rho(x)<2^{-j+1}\}, \qquad j\in \nat_0.   
\]
Moreover, $\{\varphi_j\}_{j\in \nat_0}$ satisfies the properties from Lemma \ref{zerlegungdereinsBeispiel1}. In particular,  
the functions $\varphi_j$ have finite overlap and $\varphi_j(x)=\varphi_1(2^{j-1}x)$ for all $x$ and all $j \in \N$. 
 With this we estimate 
\begin{align}\label{gut}
\|u|&\mathcal{K}^m_{a,p}(\real^d, \real^l_{\ast})\|\\
&\sim \left(\sum_{|\alpha|\leq m}\sum_{j=0}^{\infty}\int_{2^{-j-1}<\rho(x)<2^{-j+1}}
|\rho(x)^{(|\alpha|-a)}\, \partial^{\alpha}(\varphi_ju)(x)|^p\ud x\right)^{1/p}\notag\\
&\sim \left(\sum_{|\alpha|\leq m}\sum_{j=0}^{\infty}\int_{\Omega_j}2^{-(j-1)(|\alpha|-a)p}|\partial^{\alpha}(\varphi_ju)(x)|^p\ud x\right)^{1/p}\, , 
\notag
\end{align}
where for technical reasons we used $2^{-(j-1)(|\alpha|-a)p}$ instead of $2^{-j(|\alpha|-a)p}$ in the second step. 
A homogeneity argument, applied to the terms with $j\geq 1$, yields 
\begin{eqnarray*}
A& := & \left(\sum_{|\alpha|\leq m}\int_{\Omega_j}  2^{-(j-1)(|\alpha|-a)p}|\partial^{\alpha}(\varphi_ju)(x)|^p\ud x\right)^{1/p}
\\ 
& = & \left(\sum_{|\alpha|\leq m}2^{(j-1)ap}2^{-(j-1)d}\int_{\Omega_1}|\partial^{\alpha}((\varphi_ju)(2^{-j+1}\, \cdot\, ))(y)|^p\ud y\right)^{1/p}
\\
& = & 2^{(j-1)a}2^{-(j-1)d/p}\|(\varphi_ju)(2^{-j+1}\cdot)|W^m_p(\Omega_1)\|\, .
\end{eqnarray*} 
Since $\Omega_1$ has the weak cone property we conclude from  \eqref{eq-norm-sob},
\begin{eqnarray*}
A &\sim & 
\Bigg(\sum_{|\alpha|= m}2^{(j-1)ap}2^{-(j-1)d}\int_{\Omega_1}|\partial^{\alpha}((\varphi_ju)(2^{-j+1}\, \cdot\, ))(y)|^p\ud y 
 \notag
\\
&&  \qquad \quad 
+ \quad 2^{(j-1)ap}2^{-(j-1)d}\int_{\Omega_1}|(\varphi_ju)(2^{-j+1}y)|^{p}\ud y \Bigg)^{1/p}\, .
\end{eqnarray*}
Now it easy to see that the right-hand side is equivalent to
\[
\left(\sum_{|\alpha|= m}2^{(j-1)ap}\int_{\Omega_j}|2^{-(j-1)|\alpha|}\, \partial^{\alpha}(\varphi_ju)(x)|^p\ud x + 2^{(j-1)ap}\int_{\Omega_j}|(\varphi_ju)(x)|^p\ud x
\right)^{1/p}
\]
\[
\sim  \left(\sum_{|\alpha|= m}\int_{\Omega_j}|\rho(x)^{|\alpha|-a}\, \partial^{\alpha}(\varphi_ju)(x)|^p\ud x + \int_{\Omega_j}|\rho(x)^{-a}(\varphi_ju)(x)|^p\ud x
\right)^{1/p} \, . \label{est-1}
\]
On the other hand, for the term $j=0$, we easily see that 
\begin{eqnarray*}
\left(\sum_{|\alpha|\leq m}\int_{\Omega_0}|\rho(x)^{|\alpha|-a}\, \partial^{\alpha}(\varphi_0u)(x)|^p\ud x\right)^{1/p}
& \sim & \left(\sum_{|\alpha|\leq m}\int_{\Omega_0}|\partial^{\alpha}(\varphi_0u)(x)|^p\ud x\right)^{1/p}
\\
& = & \|\varphi_0 u|W^m_p(\Omega_0)\| 
\end{eqnarray*}
and 
\begin{eqnarray*}
\|\varphi_0 u|W^m_p(\Omega_0)\| 
&\sim &  \left(\sum_{|\alpha|= m}\int_{\Omega_0}|\partial^{\alpha}(\varphi_0u)(x)|^p\ud x 
+\int_{\Omega_0}|(\varphi_0u)(x)|^p\ud x
\right)^{1/p}
\\
&\sim &   \left(\sum_{|\alpha|= m}\int_{\Omega_0}|\rho(x)^{|\alpha|-a}\, \partial^{\alpha}(\varphi_0u)(x)|^p\ud x 
+\int_{\Omega_0}|\rho(x)^{-a}(\varphi_0u)(x)|^p\ud x
\right)^{1/p}, \label{est-2}
\end{eqnarray*} 
Inserting these estimates into \eqref{gut} we find 
\begin{align*}
\|u|\mathcal{K}^m_{a,p}(\real^d, \real^l_{\ast})\|&\sim  \left(\sum_{|\alpha|= m}\int_{\real^d\setminus \real^l_{\ast}}|\rho(x)^{|\alpha|-a}\, \partial^{\alpha}u(x)|^p\ud x
+ \int_{\real^d\setminus \real^l_{\ast}}|\rho(x)^{-a}u(x)|^p\ud x
\right)^{1/p}\\
&= \|u|\mathcal{K}^m_{a,p}(\real^d, \real^l_{\ast})\|_{\ast}.
\end{align*}
The proof is complete.
\end{proof}

\begin{satz}\label{norms}
Let the pair $(D,M)$ be a Lipschitz domain of polyhedral type s.t. $\Lambda_5=\emptyset$.
Furthermore, let $1\leq p<\infty$, $m\in \N$, and $a\in \R$. Then 
\[
\|u|\mathcal{K}^m_{a,p}(D,M)\|_{\ast}
 :=   \|\rho^{-a}\, u| L_{p}(D)\| +
 \sum_{|\alpha|=m}\|\rho^{m-a} \, \partial^{\alpha}u| L_p(D)\| 
\]
is an equivalent norm in $\mathcal{K}^m_{a,p}(D,M)$. 
\end{satz}

\begin{proof}
{\em Step 1}. For simplicity we deal with Case  I, the smooth cone $K$ with $M=\{0\}$. 
As mentioned before, see the list of basic properties in Subsection \ref{domains},
$C_*^\infty (K,\{0\})$ is a dense subset in $\mathcal{K}^m_{a,p}(K,\{0\})$. 
Let $u \in C_*^\infty (K,\{0\})$. Then it is readily checked that
$\mathfrak{E}u\in C_0^\infty(\R^d\setminus\{0\})$, where $\mathfrak{E}$ denotes Stein's extension operator. A closer inspection of the proof of Proposition \ref{extension} presented in \cite{Hansen} reveals that the estimates for the (weighted $L_p$-norms of) partial derivatives of $\mathfrak{E}u$ of order $m$ involve only partial derivatives of $u$ likewise of order $m$, thus we find
\[
	\|\mathfrak{E}u|\mathcal{K}^m_{a,p}(\Rd,\{0\})\|_*
		\leq c\,\|u|\mathcal{K}^m_{a,p}(K,\{0\})\|_*\,.
\]
Hence with the help of Proposition \ref{hilfe} we conclude
\begin{align*}
	\|u|\mathcal{K}^m_{a,p}(K,\{0\})\|
		&\leq\|\mathfrak{E}u|\mathcal{K}^m_{a,p}(\R^d,\{0\})\|\\
		&\leq c\,\|\mathfrak{E}u|\mathcal{K}^m_{a,p}(\R^d,\{0\})\|_*
			\leq c\,\|u|\mathcal{K}^m_{a,p}(K,\{0\})\|_*\,.
\end{align*}

{\em Step 2.}
Clearly, the Cases II and III can be handled in a similar fashion as in Step 1.
The case of a general domain $D$ of polyhedral type with singularity set $S$ can be reduced to those standard situations with the help of Lemma \ref{deco}. Note that the assumptions of Proposition \ref{extension} prevent the presence of subdomains as in Case IV.
\end{proof}


\section{Continuous embeddings}


With the help of the localization result from Subsection \ref{local1}
Sobolev-type embeddings 
for Kondratiev spaces  can now  be traced back to corresponding results for unweighted Sobolev spaces.
In a first step we deal with the model case $\Omega := \Rd  \setminus  \R_*^\ell$ and $M: = \R_*^\ell$.


\subsection{Continuous embeddings in the model case}


Again we proceed as described in Subsection \ref{general}.

\begin{prop}\label{lemma-embedding-sobolev}
Let $1 \le p \le q < \infty$, $m \in \N$,  and $a\in \R$.
Let $0 \le \ell < d$.
Then 	$\calk^m_{a,p}(\Rd, \R_*^\ell)$ is embedded into $\calk^{m'}_{a',q}(\Rd, \R_*^\ell)$ if, and only if, 
\be\label{ws-24}
		m-\frac{d}{p}\geq m'-\frac{d}{q}\qquad\text{and}\qquad
		a-\frac{d}{p} \geq a'-\frac{d}{q}\,.
\ee
\end{prop}

\begin{proof}
{\em Step 1.} We shall apply the decomposition of unity as constructed in  Proposition
	\ref{zerlegungdereins} and with the additional properties as described in Lemma \ref{zerlegungdereinsBeispiel1}. Similar to \eqref{Dj-set} we put 
\begin{equation}\label{eq-Dj}
		D_j := \{x\in\R^d:2^{-j-1} < \rho (x) <  2^{-j+1}\},\quad j\in\N_0\, , 
\end{equation}
and $\rho (x) = \min (1, |x'|)$, $x=(x',x'')$, $x' \in \R^{d-\ell}$, $x'' \in \R^\ell$.
We choose $\varepsilon_j := 2^{-j-4}$, $j\in \N_0$. 
Recall,  $\varphi_j (x) = \varphi_1 (2^{j-1}x)$, $j \in \N$. 
In view of formula \eqref{ws-20} we shall consider  the terms $\sum_{|\alpha| \le m'} 2^{-j(|\alpha|-a')q}
\int_{D_j} |\, \partial^\alpha (u \varphi_j)(x)|^q\, dx$.
For technical reasons we  replace $2^{-j(|\alpha|-a')q}$ by $2^{-(j-1)(|\alpha|-a')q}$.
A transformation of coordinates $x:= 2^{-j+1}y$ and the just mentioned homogeneity property of the system $(\varphi_j)_j$ yield
\beqq
 \sum_{|\alpha| \le m'}  && \hspace{-0.7cm} 2^{-(j-1)(|\alpha|-a')q}
\int_{D_j} |\, \partial^\alpha (u \varphi_j)(x)|^q\, dx
\\
& = & \sum_{|\alpha| \le m'}  2^{-(j-1)(|\alpha|-a')q}
\int_{D_1} |\, \partial^\alpha (u \varphi_j)(2^{-j+1}y)|^q\, 2^{(-j+1)d}\, dy
\\
& = & \sum_{|\alpha| \le m'}  2^{-(j-1)(|\alpha|-a')q} 2^{(-j+1)d}\, 2^{(j-1)|\alpha|q}\,
\int_{D_1} \Big|\, \partial^\alpha \Big(u (2^{-j+1}\, \cdot \, ) \varphi_1)(\, \cdot\, )\Big) (y)\Big|^q\,  dy
\\
& = &   2^{(j-1)a'q} 2^{(-j+1)d}\, \| \, u (2^{-j+1}\, \cdot \, )\,  \varphi_1|W^{m'}_q(D_1)\|^q\, .
\eeqq	
Here $W^{m'}_q(D_1)$ denotes the standard Sobolev space with parameters $m'$ and $q$ on $D_1$.
Clearly, 
\be\label{ws-21}
W^{m}_p(D_1) \hookrightarrow W^{m'}_q(D_1) \qquad \mbox{if} \qquad m -d/p \ge m'-d/q\, , 1 \le p \le q<\infty\, , 
\ee
see \cite[5.1]{Ada}. 
Therefore, we obtain
\beq \label{est-0}
 \sum_{|\alpha| \le m'}  && \hspace{-0.7cm} 2^{-(j-1)(|\alpha|-a')q}
\int_{D_j} |\, \partial^\alpha (u \varphi_j)(x)|^q\, dx \notag
\\
& \lesssim  &   2^{(j-1)a'q} 2^{(-j+1)d}\, \| \, u (2^{-j+1}\, \cdot \, )\,  \varphi_1|W^{m}_p(D_1)\|^q
\eeq
with hidden constants independent of $j$ and $u$.
By applying the same homogeneity arguments as above, but in reversed order, we find
\beqq
\| \, u (2^{-j+1}\, \cdot \, )\,  \varphi_1|W^{m}_p(D_1)\|^q & = &  
\Big(\sum_{|\alpha| \le m} \int_{D_1} \, | 2^{(-j+1)|\alpha|} \, \partial^\alpha  (u \varphi_j) (2^{-j+1}y)|^p \, dy \Big)^{q/p}
\\
& = &  
\Big(\sum_{|\alpha| \le m} 2^{(-j+1)|\alpha|p}\, 2^{(j-1)d} \int_{D_j} \, | \partial^\alpha  (u \varphi_j) (x)|^p \, dx \Big)^{q/p}
\eeqq
Inserting this into \eqref{est-0} we obtain
\beqq
 \sum_{|\alpha| \le m'}  && \hspace{-0.7cm} 2^{-(j-1)(|\alpha|-a')q}
\int_{D_j} |\, \partial^\alpha (u \varphi_j)(x)|^q\, dx
\\
& \lesssim &
 2^{(j-1)a'q} 2^{(-j+1)d}\,
\Big(\sum_{|\alpha| \le m} 2^{(-j+1)|\alpha|p}\, 2^{(j-1)d} \int_{D_j} \, | \partial^\alpha  (u \varphi_j) (x)|^p \, dx \Big)^{q/p}
\\
& \lesssim &
 \Big(
\sum_{|\alpha| \le m} 2^{(j-1)( \frac dp - \frac dq-|\alpha| + a')p} \int_{D_j} \, | \partial^\alpha  (u \varphi_j) (x)|^p \, dx \Big)^{q/p}
\\
& \lesssim &
 \Big(
\sum_{|\alpha| \le m} 2^{(j-1)( \frac dp - \frac dq + a'-a)p} 
\int_{D_j} \, |2^{-j(|\alpha|-a))} \partial^\alpha  (u \varphi_j) (x)|^p \, dx \Big)^{q/p}\, .
\eeqq
Obviously $2^{-j(|\alpha|-a)}	\asymp \rho(x)^{|\alpha|-a}$ on $D_j$.
By assumption 
$ \frac dp - \frac dq + a'-a \le 0 $.
Hence
\beqq
\sum_{|\alpha| \le m'} && \hspace{-0.7cm} \sum_{j=1}^\infty 2^{-j(|\alpha|-a')q}
\int_{D_j} |\, \partial^\alpha (u \varphi_j)(x)|^q\, dx
\\
& \lesssim &
\sum_{j=1}^\infty  \Big(\sum_{|\alpha| \le m}  
\int_{D_j} \, |\, \rho (x)^{|\alpha|-a} \partial^\alpha  (u \varphi_j) (x)|^p \, dx \Big)^{q/p}
\eeqq
Next we shall use that $\ell_1 \hookrightarrow \ell_{q/p}$. This yields
\beqq
\sum_{|\alpha| \le m'} && \hspace{-0.7cm} \sum_{j=1}^\infty 2^{-j(|\alpha|-a')q}
\int_{D_j} |\, \partial^\alpha (u \varphi_j)(x)|^q\, dx
\\
\\
& \lesssim & \Big(\sum_{j=1}^\infty  \sum_{|\alpha| \le m}  
\int_{D_j} \, |\, \rho (x)^{|\alpha|-a} \partial^\alpha  (u \varphi_j) (x)|^p \, dx \Big)^{q/p}\, .
\eeqq
For the term with $j=0$ it is enough to apply the Sobolev embedding \eqref{ws-21} with $D_1$ replaced by $D_0$.
In view of \eqref{ws-20} this proves sufficiency of our conditions.
\\
{\em Step 2.} Necessity. The necessity of 
$m -d/p \ge m'-d/q$ is part of the classical Sobolev theory, we refer to \cite[5.2.4]{Ada}.
It remains to prove the necessity of 
$a-d/p \ge a'-d/q$. Therefore we choose a non-trivial function 
$u \in C_0^\infty (\Rd)$ such that $\supp u \subset \{x\in \Rd:~ 0<|x'|<1\}$.
Such  a function and all dilated versions $u(\lambda\, \cdot )$, $\lambda >0$,  belong to all spaces  
$\calk^m_{a,p}(\Rd, \R_*^\ell)$.  
Observe the following homogeneity property for values $\lambda>1$:
\beq\label{ws-23}
\| \, u(\lambda\, \cdot \, )|\calk^m_{a,p}(\Rd, \R_*^\ell)\|^p & = &  
\sum_{|\alpha| \le m}  \int_{\Rd} \lambda^{|\alpha|p}\, 
|\rho(x)^{|\alpha|-a} \, \partial^\alpha u (\lambda x)|^p\, dx
\nonumber
\\
 & = &  
\sum_{|\alpha| \le m}  \lambda^{|\alpha|p-d}\, \int_{\supp u} 
|\rho(y/\lambda)^{|\alpha|-a} \, \partial^\alpha u (y)|^p\, dy
\nonumber
\\
& \asymp &  
\sum_{|\alpha| \le m}  \lambda^{ap - d }\, \int_{\supp u} 
|\rho(y)^{|\alpha|-a} \, \partial^\alpha u (y)|^p\, dy
\nonumber
\\
& = & \lambda^{ap - d } \, \| \, u \, |\calk^m_{a,p}(\Rd, \R_*^\ell)\|^p \, .
\eeq
We assume $\calk^m_{a,p}(\Rd, \R_*^\ell) \hookrightarrow \calk^{m'}_{a',q}(\Rd, \R_*^\ell)$.
This implies the existence of a positive constant $c$ such that 
\[
 \| \, u(\lambda\, \cdot \, )|\calk^{m'}_{a',q}(\Rd, \R_*^\ell)\| \le c\, \| \, u(\lambda\, \cdot \, )|\calk^m_{a,p}(\Rd, \R_*^\ell)\|
\]
 holds for all $\lambda \ge 1$ and in view of \eqref{ws-23} $a'-d/q \le a-d/p$ as claimed.
\end{proof}

The counterpart for $q=\infty$ can be formulated as follows.

\begin{prop}\label{lemma-embedding-sobolevinfty}
Let $a\in \R$, $m \in \N$, and $0 \le \ell < d$.
\\
{\rm (i)} Let $1 < p  < \infty$. 
Then 	$\calk^m_{a,p}(\Rd, \R_*^\ell)$ is embedded into $\calk^{m'}_{a',\infty}(\Rd, \R_*^\ell)$ if, and only if, 
\[
m-\frac{d}{p} > m'\qquad\text{and} \qquad a-\frac{d}{p} \geq a'\,.
\]
{\rm (ii)} Let $ p = 1$. 
Then 	$\calk^m_{a,1}(\Rd, \R_*^\ell)$ is embedded into $\calk^{m'}_{a',\infty}(\Rd, \R_*^\ell)$ if, and only if, 
\[
m-{d} \ge  m'\qquad\text{and} \qquad a-d \geq a'\,.
\]
\end{prop}


\subsection{Continuous  embeddings for Kondratiev spaces on domains of generalized  polyhedral type}


\begin{satz}\label{sobolev}
Let $1 \le p \le q < \infty$, $m \in \N$,  and $a\in \R$.
Let the pair $(D,M)$ be of generalized polyhedral type.
Then 	$\calk^m_{a,p}(D,M)$ is embedded into $\calk^{m'}_{a',q}(D,M)$ if, and only if, 
\[
		m-\frac{d}{p}\geq m'-\frac{d}{q}\qquad\text{and}\qquad
		a-\frac{d}{p} \geq a'-\frac{d}{q}\,.
\]
\end{satz}

\begin{proof}
{\em Step 1.} Let  $(D,M)$  be as in Cases I-III.
The linear and continuous extension operator $\mathfrak{E}$ in
Proposition \ref{extension} allows to reduce sufficiency 
to Proposition \ref{lemma-embedding-sobolev}.
Concerning necessity we mention that our argument in Proposition \ref{lemma-embedding-sobolev} relied on a function $u\in C_0^\infty (\Rd)$
with $\supp u \subset \{x\in \Rd:~ 0<|x'|<1\}$.
For notational convenience we consider Case I with $D=K$ and $M=\{0\}$.
Let $u\in C_0^\infty(K)$. Again all functions $u(\lambda \, \cdot\, )$, $\lambda >1$,
belong to $C_0^\infty (K)$. Now, similar as in the proof of Theorem \ref{norms} we can reduce everything to Proposition \ref{lemma-embedding-sobolev}.
\\
{\em Step 2.}
Let  $D$  be a polyhedral cone as in Case IV.  
It is enough to combine  Lemma \ref{deco2} with Step 1.
\\
{\em Step 3.} Let $D$ be a domain of polyhedral type with singularity set $S$.
Then we make use of Lemma \ref{deco} to reduce the problem to an application of Steps 1 and 2.
\\
{\em Step 4.} In all remaining cases we apply Lemma \ref{deco3} with Steps 1 and 2. 
\end{proof}

Arguing as in case $q<\infty$ we can derive also the following for $q= \infty$.

\begin{satz}\label{sobolevinfty}
Let $a\in \R$ and  $m \in \N$.
Let the pair $(D,M)$ be of generalized polyhedral type.
\\
{\rm (i)} Let $1 < p  < \infty$. 
Then 	$\calk^m_{a,p}(D,M)$ is embedded into $\calk^{m'}_{a',\infty}(D,M)$ if, and only if, 
\[
m-\frac{d}{p} > m'\qquad\text{and} \qquad a-\frac{d}{p} \geq a'\,.
\]
{\rm (ii)} Let $ p = 1$. 
Then 	$\calk^m_{a,1}(D,M)$ is embedded into $\calk^{m'}_{a',\infty}(D,M)$ if, and only if, 
\[
m-{d} \ge  m'\qquad\text{and} \qquad a-d \geq a'\,.
\]
\end{satz}

\begin{rem}
 \rm Embeddings of Kondratiev spaces have been proved also in 
Maz'ya and Rossmann \cite{MR}. 
We refer to Lemma 1.2.2 and Lemma  1.2.3 (smooth cones),  Lemma 2.1.1 (dihedron), Lemma 3.1.3 and Lemma 3.1.4 (cones with edges)
 and Lemma 4.1.2 (domains of polyhedral type).
Only sufficiency is discussed there. Except for smooth cones  the case of equality of $m-d/p$ and $m'-d/q$
is always excluded. 
\end{rem}


\section{Compact  embeddings}


Having dealt with continuous embeddings within the scale of Kondratiev spaces so far, 
we now investigate when these embeddings are compact.
Roughly speaking, it turns out that whenever we deal with strict inequalities in 
Theorems \ref{sobolev} and \ref{sobolevinfty} we obtain compact embeddings.
Recall that   in a pair $(D,M)$ of generalized  polyhedral type $D$ is a bounded domain.

\begin{satz}\label{sobolevcompact}
Let $1 \le p \le q \le \infty$, $m \in \N$, and $a\in \R$.
Let $(D,M)$ be either $(K,M)$ (Case I) or $(P,\Gamma)$ (Case II) or $(I,M_\ell)$ (Case III) or $(Q,M)$ (Case IV)  or a domain of polyhedral type
with $S$ being the  singularity set of $D$.
Then 	$\calk^m_{a,p}(D,M)$ is compactly  embedded into $\calk^{m'}_{a',q}(D,M)$ if, and only if, 
\[
		m-\frac{d}{p} > m'-\frac{d}{q}\qquad\text{and}\qquad
		a-\frac{d}{p} > a'-\frac{d}{q}\,.
\]
\end{satz}

\begin{proof}
{\em Step 1.} Sufficiency. Here we will  follow Maz'ya, Rossmann \cite[Lemma 4.1.4]{MR}.
\\
We fix $\varepsilon >0$.
Moreover, for $\delta >0$ we put 
\[
 D_\delta := \{x \in D: \quad \rho (x) <  \delta\}\,.
\]
If $\delta $ is small enough, then, as $D$ itself, $D \setminus \overline{D}_\delta $ has the cone property.
This implies the compactness of the embedding $W^m_{p}(D \setminus \overline{D}_\delta) \hookrightarrow \hookrightarrow W^{m'}_{q}(D \setminus \overline{D}_\delta)$
and therefore (with a slight abuse of notation)
$\calk^m_{a,p}(D \setminus \overline{D}_\delta,M) \hookrightarrow \hookrightarrow \calk^{m'}_{a',q}(D \setminus \overline{D}_\delta,M)$.

Let $U$ denote the unit ball in  $\calk^m_{a,p}(D,M)$. Then this compact embedding, together with
\[
	\sup_{u\in U}\|u|\calk^{m'}_{a',q}(D\setminus\overline{D}_\delta,M)\|
		\leq\sup_{u\in U}\|u|\calk^{m'}_{a',q}(D,M)\|=C_1<\infty
\]
which in turn is a consequence of the continuity of the embedding
$\calk^m_{a,p}(D,M) \hookrightarrow \calk^{m'}_{a',q}(D,M)$ (see Theorems  \ref{sobolev} and  \ref{sobolevinfty}), implies the existence of a
finite $\varepsilon$-net $u_1, \ldots, u_N \in U$ such that for all  
$u \in U$ we have
\[
 \min_{i=1, \ldots \, , N} \, \|\, u-u_i \, |\calk^{m'}_{a',q}(D \setminus \overline{D}_\delta,M)\|< \varepsilon\, .
\]
Next we define  $\sigma := a-\frac dp + \frac dq -a'$. By assumption $\sigma >0$.
If $u \in U $, applying Theorems  \ref{sobolev} and \ref{sobolevinfty},   we conclude
 \beqq
 \|u|\calk^{m'}_{a',q}(D_\delta ,M)\| & = &  \Big(\sum_{|\alpha|\leq m'} \int_{D_\delta}
|\rho(x)^{|\alpha|- (a - \frac dp + \frac dq -\sigma)} \partial^\alpha u(x)|^q\,dx\Big)^{1/q}
\\
& \le &
 \delta ^\sigma\,    \Big(\sum_{|\alpha|\leq m'} \int_{D_\delta}
|\rho(x)^{|\alpha|- (a - \frac dp + \frac dq)} \partial^\alpha u(x)|^q\,dx\Big)^{1/q}
\\
& \le  &
 \delta ^\sigma\,    \|  \, u \, | \calk^{m'}_{a-\frac dp + \frac dq,q}(D,M)\|
\\
& \le & \delta ^\sigma\,   C_1\,  \|  \, u \, | \calk^{m}_{a,p}(D,M)\| 
\\
& \le & \delta^\sigma \, C_1\, .
\eeqq
Choosing $\delta$ so small such that $\delta^\sigma \, C_1< \varepsilon$ we get

\[
 \min_{i=1, \ldots \, , N} \, \|\, u-u_i \, |\calk^{m'}_{a',q}(D,M)\| < 3\, \varepsilon\, .
\] 
Hence, the embedding is compact.
\\
{\em Step 2.} We deal with necessity.\\
{\em Substep 2.1.} The necessity of $m-\frac{d}{p} > m'-\frac{d}{q}$ 
 follows from the necessity of this condition for the compactness of the embedding 
$W^m_{p}(\Omega) \hookrightarrow \hookrightarrow W^{m'}_{q}(\Omega)$, where $\Omega$ is a domain in ${\R^d}$ satisfying a cone condition.
It is enough to choose $\Omega$ as a ball contained in $D$ such that $\dist (\Omega,M)\in [A,B]$ for $0 < A < B < \infty$.
\\
{\em Substep 2.2.} Necessity of $a-\frac{d}{p} > a'-\frac{d}{q}$. 
Let $u \in C_0^\infty (\Rd)$ be a nontrivial function such that 
$\supp u \subset B_1(0)$. Next we select a sequence $(x^j)_j \subset D$ such that 
\[
B_{2^{-(j+4)}}(x^j) \, {\subset}\,  \Big\{x \in D: \: 2^{-j} <  \rho(x)< 2^{-j+1} \Big\}\, , \qquad j \ge j_0 (D).
\]
For 
$D$ being  a domain of generalized polyhedral type  it is clear that such a sequence exists if $j_0 (D)$ is chosen sufficiently large.
We put
\[
 u_j (x):= u(2^j (x-x^j)), \quad  x \in \Rd, \quad j \ge j_0 (D).
\]
It follows 
\beqq
\|\, u_j\, |\calk^m_{a,p}(D,M)\|^p & \asymp &  \sum_{|\alpha|\leq m}   2^{-j(|\alpha|-a)p} \int_D
|\partial^\alpha u_j(x)|^p\, dx
\\
&\asymp & \sum_{|\alpha|\leq m}   2^{-j(|\alpha|-a)p} 2^{j|\alpha|p} \, \int_{B_1(0)} |\partial^\alpha u(y)|^p\, 2^{-jd}\, dy
\\
&\asymp & 2^{j(ap-d)}  \,  \|\, u\, |W^m_p (B_1(0))\| \, .
\eeqq
Hence, we obtain  
\[
\|\, 2^{-j(a- \frac dp)}\,  u_j\, |\calk^m_{a,p}(D,M)\|\asymp 1\, , \qquad j \ge j_0 (D)\, . 
\]
Let $a'':= a-\frac{d}{p} + \frac{d}{q}$.
Then
\[
 \|\, 2^{-j(a''-\frac dq)}\,  u_j\, |\calk^m_{a'',q}(D,M)\|  = \|\, 2^{-j(a-\frac dp)}\,  u_j\, |\calk^m_{a'',q}(D,M)\|\geq c>0\, ,  
\qquad j \ge j_0 (D)\, .
\]
Observe that 
\[
 \supp u_j \, \cap\, \supp u_{j+2} = \emptyset\, , \qquad j \ge j_0 (D).
\]
Consequently $\big({2^{-2j(a-\frac dp)}}u_{2j}\big)_j$ is a bounded sequence in  $\calk^m_{a,p}(D,M)$ which does not have a convergent subsequence in $\calk^m_{a'',q}(D,M)$.
\end{proof}

\begin{rem}\rm
For Sobolev spaces  this result, usually called Rellich-Kondrachov theorem, has been known for a long time, 
we refer to Adams \cite[Thm.~6.2]{Ada}.
The sufficiency part of Theorem \ref{sobolevcompact}  is essentially contained in Maz'ya, Rossmann \cite[Lemma 4.1.4]{MR}.
\end{rem}


\section{Pointwise multiplication in Kondratiev spaces}
\label{pm}



\subsection{Pointwise multiplication in the model case }
\label{pm1}


First we deal with our model case $\calk^{m}_{a,p} (\R^d,\R^\ell_*)$. 
As before the main idea consists in tracing everything back to the standard Sobolev case.


\subsubsection{Algebras with respect to pointwise multiplication}


Recall that $W^m_p (\Rd)$, $m \in \N$, $1\le p < \infty$, is an algebra with respect to pointwise multiplication
if, and only if, either $1<p<\infty$ and $m>d/p$ or $p=1$ and $m \ge d$, cf.  \cite[Thm.~5.23]{Ada} and also \cite[Sect.~6.1]{MS85} 
for the limiting case. In particular, 
\be\label{sobo1}
\|u \, \cdot \,v|W^m_p(\Rd)\|\leq c \|u|W^m_p(\Rd)\|\|v|W^m_p(\Rd)\|, 
\ee
for all $u,v\in W^m_p(\Rd)$ and {some} $c>0$. 
Observe, that these conditions are equivalent with the $L_\infty$-embedding of the Sobolev space.
\\
As a first step we now consider the following more general estimate of a product.
The strategy of the proof will be essential for the following results.

\begin{satz}\label{mult1}
Let $1\le p < \infty$, $m \in \N$, $0 \le \ell < d$, and $a \in \R$. In case that $W^m_p (\Rd)$ is an algebra with respect to pointwise 
multiplication  there exists a constant $c$ s.t.
	\begin{equation*}
		\|\, u\, \cdot \, v\,  |\calk^{m}_{2a-\frac{d}{p},p} (\R^d,\R^\ell_*)
			\|\le c \, \|\, u\, |\calk^{m}_{a,p} (\R^d,\R^\ell_*)\| \, \|\, v\,  |\calk^{m}_{a,p} (\R^d,\R^\ell_*)\|
	\end{equation*}
	holds for all $u, v \in   \calk^{m}_{a,p} (\R^d,\R^\ell_*)$. 
\end{satz}

\begin{proof}
{\em Step 1.} Some preliminary estimates. Recall the definition of the sets $D_j$ in \eqref{eq-Dj} and the weight function $\rho$ from \eqref{eq-weight-model}.
Fix $j\geq 1$. We start with
\begin{align*}
\int_{D_j} | \partial^\beta w (x)|^p\, dx
		&\lesssim 2^{-dj}\, \int_{D_1} |\,\partial^{\beta} w (2^{-j+1}y)\,|^p\, dy\\
			& \lesssim 2^{j|\beta|p}\,  2^{-dj}\, \int_{D_1}
				\bigl| \partial^{\beta}\bigl(w (2^{-j+1} \, \cdot \,)\bigr) (y)\bigr|^p\, dy\,,
	\end{align*}
where we used the transformation of coordinates $y:= 2^{j-1}x$. Summation over $\beta$ then implies
	\begin{equation}\label{ws-060}
		\sum_{|\beta|\leq m} 2^{-j|\beta|p} \int_{D_j} | \partial^\beta w (x)|^p\, dx
			\lesssim    2^{-dj}\, \| \,w (2^{-j+1} \, \cdot \,)\, |W^m_p (D_1)\|^p.
	\end{equation}
We further note
\begin{eqnarray}
\label{ws-070}
\| \,w (2^{-j+1} \,\cdot\, )\, |W^m_p (D_1)\|^p
&= & 
\sum_{|\beta| \leq m} 2^{-(j-1)|\beta|p} \, \int_{D_1} |\partial^\beta w (2^{-j+1}y)  |^p\, dy
\nonumber
\\
&=& \sum_{|\beta| \leq m} 2^{-(j-1)|\beta|p} \, \int_{D_j} |\partial^\beta w (x)  |^p\, 2^{d(j-1)} \, dx\,.
\end{eqnarray}
{\em Step 2.} Let $u, v\in \calk^{m}_{a,p} (\Rd,\R^\ell_*)$ such that 
\[
\supp u,\, \supp v \subset \{x=(x',x'')\in \Rd:~~|x'| \le {3/4}\}\, .
\]
Since $W^m_p (\Rd)$ is an algebra,  \eqref{ws-060} applied to $w=u\, \cdot \,v$,  leads to
\begin{align}\label{w-5}
\|\, u\, \cdot \,v &  |\calk^{m}_{a,p} (\R^d,\R^\ell_*)\|^p
\nonumber
\\
& \lesssim   \sum_{j=1}^\infty \sum_{|\beta| \le m} 2^{-j(|\beta|-a)p}
\int_{D_j} | \,  \partial^\beta (u\, \cdot \,v) (x)|^p\, dx 
\nonumber
\\
&\lesssim  \sum_{j=1}^\infty 2^{jap} \, 2^{-dj}\,
\| \,u (2^{-j+1} \, \cdot \, ) v (2^{-j+1} \, \cdot \, )\, |W^m_p (D_1)\|^p
\nonumber
\\
&\lesssim  \sum_{j=1}^\infty 2^{jap} \, 2^{-dj}\,
\| \,u (2^{-j+1} \, \cdot \, )\, |W^m_p (D_1)\|^p \, 
	\|\, v (2^{-j+1} \, \cdot \, )\, |W^m_p (D_1)\|^p\,.
\end{align}
In view of \eqref{ws-070} we finally find
\begin{align*}
		\|\, u\, \cdot \, v  |\calk^{m}_{a,p} (\R^d,\R^\ell_*)\|^p
			& \lesssim \sum_{j=1}^\infty 2^{jap} \, 2^{-dj} \biggl(\sum_{|\beta| \le m} 
				2^{-j|\beta|p} \, 2^{dj} \, \int_{D_j} | \, \partial^\beta u (x)  |^p dx\biggr)
			\\
&\qquad\times \quad \biggl(\sum_{|\beta| \le m} 2^{-j|\beta|p} \, 2^{dj} \, 
\int_{D_j} | \, \partial^\beta v (x)  |^p\,  dx\biggr)
\\
& \lesssim \sum_{j=1}^\infty \biggl(\sum_{|\beta| \le m} 
				2^{-j|\beta|p} \, 2^{j(\frac a2 + \frac{d}{2p})p} \, 
				\int_{D_j}| \, \partial^\beta u (x)  |^p\, dx\biggr)
			\\
			& \qquad\times \quad \sup_{j\ge 1} \biggl(\sum_{|\beta| \le m} 
				2^{-j|\beta|p} \, 2^{j(\frac a2 + \frac{d}{2p})p} \, 
				\int_{D_j} |\, \partial^\beta v (x)  |^p\,  dx\biggr)
			\\
			& \lesssim  \|\, u\, |\calk^{m}_{\frac a2 + \frac{d}{2p},p} (\R^d,\R^\ell_*)\|^p\ 
				\|\,  v\,  |\calk^{m}_{\frac a2 + \frac{d}{2p},p} (\R^d,\R^\ell_*)\|^p
	\end{align*}
{\em Step 3.} 
Let  $u, v\in \calk^{m}_{a,p} (\Rd,\R^\ell_*)$ such that 
\[
\supp u,\, \supp v \subset \{x=(x',x'')\in \Rd:~~|x'| \ge 1/4\}\, .
\]
In this situation the weight does not play any role and we may apply the result for Sobolev spaces directly.
\\
{\em Step 4.} Let  $u, v\in \calk^{m}_{a,p} (\Rd,\R^\ell_*)$. There  exists a smooth function $\eta \in C^m (\Rd)$ with the following properties:
$\eta (x) = 1$ if $|x'|\le 1/2$ and $\supp \eta \subset \{x\in \Rd:~~|x'|\le 3/4\}$.
Let $\tau \in C^m (\Rd)$ be a function such that $\tau = 1$ on $\supp (1-\eta^2)$ and $\supp \tau \subset \{x \in \Rd:~~|x'|\ge 1/4\}$.
Making use of the basic properties of the Kondratiev spaces listed in Subsection \ref{sub1} and the results of Steps 2 and 3, we obtain
\begin{align*}
\|\, u\, \cdot &\, v\,   |\calk^{m}_{a,p} (\R^d,\R^\ell_*)\|^p \\
&= 
\|\, u\, \cdot \, v \, \cdot \, \eta^2 + u \, \cdot \, v \, \cdot \, (1-\eta^2)\,  |\calk^{m}_{a,p} (\R^d,\R^\ell_*)\|^p 
\\
& \lesssim   \|\, (u \, \cdot \, \eta) \, \cdot \, (v \, \cdot \,\eta) \, |\calk^{m}_{a,p} (\R^d,\R^\ell_*)\|^p 
+ \|\, (u \, \cdot \,(1-\eta^2)) \, \cdot \, (v \, \cdot \, \tau) \, |\calk^{m}_{a,p} (\R^d,\R^\ell_*)\|^p
\\
& \lesssim   \|\, u \, \cdot \,\eta  \, |\calk^{m}_{{\frac a2+\frac{d}{2p}},p} (\R^d,\R^\ell_*)\|^p 
\| \, v\, \cdot \, \eta \, |\calk^{m}_{{\frac a2+\frac{d}{2p}},p} (\R^d,\R^\ell_*)\|^p
\\
& \qquad \quad 
+ \|\, u \, \cdot \,(1-\eta^2) \,  |\calk^{m}_{{\frac a2+\frac{d}{2p}},p} (\R^d,\R^\ell_*)\|^p 
\|\,  v\, \cdot \, \tau \, |\calk^{m}_{{\frac a2+\frac{d}{2p}},p} (\R^d,\R^\ell_*)\|^p
\\
& \lesssim   \|\, u \,  |\calk^{m}_{{\frac a2+\frac{d}{2p}},p} (\R^d,\R^\ell_*)\|^p 
\| \, v\,  |\calk^{m}_{{\frac a2+\frac{d}{2p}},p} (\R^d,\R^\ell_*)\|^p\, .
\end{align*}
The proof is complete.
\end{proof}

\begin{cor}\label{mult11}
Let $1\le p < \infty$, $m \in \N$, $a \in \R$, and $0 \le \ell < d$. 
\\
{\rm (i)} Let $a\ge\frac{d}{p}$ and  either $1<p<\infty$ and  $m>d/p$ or $p=1$ and $m\ge d$. Then
the Kondratiev space $\calk^{m}_{a,p} (\Rd,\R^\ell_*)$ is an algebra with respect to pointwise multiplication, i.e., there exists a constant $c$ such that
\[
\|u\, \cdot \,v|\calk^{m}_{a,p} (\Rd,\R^\ell_*)\|\leq c \|u|\calk^{m}_{a,p} (\Rd,\R^\ell_*)\|\|v|\calk^{m}_{a,p} (\Rd,\R^\ell_*)\|
\]
holds for all $u,v\in \calk^{m}_{a,p} (\Rd,\R^\ell_*)$.
\\
{\rm (ii)} Let $\ell =0$. Then the Kondratiev space $\calk^{m}_{a,p} (\Rd,\R^0_*)$ is an algebra with respect to pointwise multiplication
if, and only if, $a\ge\frac{d}{p}$ and  either $1<p<\infty$ and  $m>d/p$ or $p=1$ and $m\ge d$. 
\end{cor}

\begin{proof}
{\em Step 1.} Sufficiency.
As mentioned in Subsection \ref{sub1}, the spaces $\calk^{m}_{a,p} (\Omega,M)$ are monotone in $a$.  Since 
\[
\calk^{m}_{2a-\frac{d}{p},p} (\R^d,\R^\ell_*)
\hookrightarrow \calk^{m}_{a,p} (\R^d,\R^\ell_*) \qquad \mbox{if}\qquad a \ge\frac{d}{p}\, , 
\]
the claim follows from Theorem \ref{mult1}.
\\
{\em Step 2.} Necessity in case $\ell=0$. The necessity of the conditions $1<p<\infty$ and  $m>d/p$ or $p=1$ and $m\ge d$ can be reduced to the necessity 
in case of Sobolev spaces by using an obvious cut-off argument.
It remains to prove the necessity of $a \ge d/p$. Therefore we construct a counterexample in case $a < d/p$.
Employing Lemma \ref{drei} {in the Appendix} we conclude
\[
 \tilde{\varrho}^{b} \, \cdot \, \psi \in \calk^{m}_{a,p} (\R^d,\R^0_*) \qquad \text{if, and only if,} \qquad a-b < \frac dp
\]
as well as 
\[
 (\tilde{\varrho}^{b} \, \cdot \, \psi)^2 \in \calk^{m}_{a,p} (\R^d,\R^0_*) \qquad \text{if, and only if,}  \qquad a-2b < \frac dp \, .
\]
We choose $b< 0$ such that 
\[
 a -\frac  dp < b < \frac{a-d/p}{2} \, . 
\]
Then $ \tilde{\varrho}^{b} \, \cdot \, \psi \in \calk^{m}_{a,p} (\R^d,\R^0_*)$ but 
$(\tilde{\varrho}^{b} \, \cdot \, \psi)^2$ does not belong to it.
\end{proof}

As  the proof of Theorem \ref{mult1} shows, we also have the following slightly more general version.

\begin{cor}\label{lem-mult1}
	Let $1\le p < \infty$, $m,m_1,m_2 \in \N_0$, $a_1,a_2\in \R$, and $0 \le \ell < d$. If either 
$1< p < \infty$ and $\min(m_1,m_2)\geq m>d/p$ or $p=1$ and $\min(m_1,m_2)\geq m \ge d$, then there
	exists a constant $c$ s.t.
	\begin{equation}\label{ws-08}
		\|\, u\, \cdot \, v\,  |\calk^{m}_{a_1+a_2-\frac{d}{p},p} (\Rd,\R^\ell_*)\|
			\le c \, \|\, u\, |\calk^{m_1}_{a_1,p} (\Rd,\R^\ell_*)\| \, \|\, v\,  |\calk^{m_2}_{a_2,p} (\Rd,\R^\ell_*)\|
	\end{equation}
	holds for all $u\in\calk^{m_1}_{a_1,p}(\Rd,\R^\ell_*)$ and $v \in   \calk^{m_2}_{a_2,p} (\Rd,\R^\ell_*)$. 
\end{cor}


\subsubsection{Multiplication with unbounded functions}


There are two known possibilities to extend \eqref{sobo1} to unbounded functions.
The first one is as follows. 
Let $1 <p< \infty$, $m_0 \in \N$, $m_1 \in \N_0$, and 
\be\label{sobo3}
\frac{d}{2p}  \le m_0 < \frac dp\, .
\ee 
Then there exists a constant $c>0$ such that
\be\label{sobo2}
\|u\, \cdot \,v|W^{m_1}_p(\Rd)\|\leq c \|u|W^{m_0}_p(\Rd)\|\, \|v|W^{m_0}_p(\Rd)\|, 
\ee
holds for all $u,v\in W^{m_0}_p(\Rd)$, where 
\be\label{sobo4}
 m_1 \le 2m_0 -\frac dp\, .
\ee
Observe that these restrictions are natural in such a context. To see this one may use the following family of test functions:
\[
 f_\alpha (x):= |x-x^0|^{-\alpha}\, \psi (x-x^0)\, , \qquad x \in \Rd, \quad \alpha >0\, .
\]
Here $x^0$ is an arbitrary point in $\Rd$ and $\psi \in C_0^\infty (\Rd)$ such that $\psi (0)>0$.
Elementary calculations yield $f_\alpha \in W^m_p (\Rd)$ if, and only if, $\alpha < \frac dp - m$, cf. \cite[Lemma~2.3.1/1]{RS}.
First we comment on the lower bound in \eqref{sobo3}.
If $m_0 = d/(2p)-\varepsilon$ for some $\varepsilon >0$, then we may choose $\alpha = d/(2p)$ getting $f_\alpha \in W^{m_0}_p (\Rd)$. 
But the product $f_\alpha \, \cdot \, f_\alpha $ does not belong to $L_p (\Rd)$
since the order of the singularity is $d/p$.\\
Secondly, we deal with \eqref{sobo4}.
Let $m_0$ satisfy \eqref{sobo3} and choose $\alpha = \frac dp - m_0 - \varepsilon$ with $\varepsilon>0$ small.
For the  product $f_\alpha \, \cdot \, f_\alpha $ we conclude that it belongs to $W^{m_1}_p (\Rd)$ if
\[
m_1 < \frac dp - 2\alpha = - \frac dp + 2 m_0 + 2\varepsilon\, .
\]
For $\varepsilon \downarrow 0$ the restriction \eqref{sobo4} follows.
\\
For all this we refer to \cite{RS}, in particular Corollary 4.5.2.

In terms of Kondratiev spaces the following can be shown.

\begin{satz}\label{mult111}
Let $1< p < \infty$, $m_0 \in \N$, $m_1 \in \N_0$, $0 \le \ell < d$, and $a_0 \in \R$. 
Suppose 
\[
\frac{d}{2p}  \le m_0 < \frac dp\, , \qquad 
m_1 \le 2\, m_0 -\frac dp, \qquad \mbox{and}\qquad a_1:= 2\Big(a_0 - \frac{d}{2p} \Big). 
\]
Then there exists a constant $c$ s.t.
\begin{equation}\label{sobo14}
\|\, u\, \cdot \, v\,  |\calk^{m_1}_{a_1,p} (\R^d,\R^\ell_*)
\|\le c \, \|\, u\, |\calk^{m_0}_{a_0,p} (\R^d,\R^\ell_*)\| \, \|\, v\,  |\calk^{m_0}_{a_0,p} (\R^d,\R^\ell_*)\|
\end{equation}
	holds for all $u, v \in   \calk^{m_0}_{a_0,p} (\R^d,\R^\ell_*)$. 
\end{satz}

\begin{proof}
We follow the same strategy as used in the proof of Theorem \ref{mult1}.
Therefore only some comments to the modification are made.
We have to modify \eqref{w-5}. Using \eqref{sobo2} instead of \eqref{sobo1} we find
\begin{eqnarray*}
\|\, u\,\cdot\,  v\, && \hspace{-0.7cm} |\calk^{m_1}_{a_1,p} (\R^d,\R^\ell_*)\|^p
\\
&\lesssim & \sum_{j=1}^\infty 2^{ja_1p} \, 2^{-dj}\,
\| \,u (2^{-j+1} \, \cdot \, )\, |W^{m_0}_p (D_1)\|^p \, 
	\|\, v (2^{-j+1} \, \cdot \, )\, |W^{m_0}_p (D_1)\|^p\,.
\end{eqnarray*}
 Now we continue as before and obtain
\begin{align*}
\|\, u\, \cdot \, v\,  |\calk^{m_1}_{a_1,p} (\R^d,\R^\ell_*)\|^p
& \lesssim \sum_{j=1}^\infty 2^{ja_1p} \, 2^{-dj} \biggl(\sum_{|\beta| \le m_0} 
2^{-j|\beta|p} \, 2^{dj} \, \int_{D_j} | \, \partial^\beta u (x)  |^p dx\biggr)
\\
&\qquad\times \quad \biggl(\sum_{|\beta| \le m_0} 2^{-j|\beta|p} \, 2^{dj} \, 
\int_{D_j} | \, \partial^\beta v (x)  |^p\,  dx\biggr)
\\
& \lesssim \sum_{j=1}^\infty \biggl(\sum_{|\beta| \le m_0} 
2^{-j|\beta|p} \, 2^{j(\frac {a_1}{2} + \frac{d}{2p})p} \, 
\int_{D_j}| \, \partial^\beta u (x)  |^p\, dx\biggr)
\\
& \qquad\times \quad \sup_{j\ge 0} \biggl(\sum_{|\beta| \le m_0} 
2^{-j|\beta|p} \, 2^{j(\frac {a_1}{2} + \frac{d}{2p})p} \, 
\int_{D_j} |\, \partial^\beta v (x)  |^p\,  dx\biggr)
\\
& \lesssim  \|\, u\, |\calk^{m_0}_{\frac {a_1}{2} + \frac{d}{2p},p} (\R^d,\R^\ell_*)\|^p\ 
\|\,  v\,  |\calk^{m_0}_{\frac {a_1}{2} + \frac{d}{2p},p} (\R^d,\R^\ell_*)\|^p\, .
\end{align*} 
 This proves the claim.
\end{proof}

\begin{rem}
 \rm
Observe that we have to accept a loss in regularity in \eqref{sobo14} since, by \eqref{sobo3} we always have
$m_1 < m_0$.
\end{rem}

As before there exists a generalization in the dependence with respect to the parameter $a$.
It will be based on the following product estimate in case of Sobolev spaces, see Corollary 4.5.2 in \cite{RS}. 
Let $1 <p< \infty$, $m_0, m_1 \in \N$, $m_2 \in \N_0$, 
\be\label{sobo10}
\frac{d}{p} \le m_0 + m_1, \qquad \mbox{and}\qquad  m_0, m_1 < \frac dp\, .
\ee 
Then there exists a constant $c>0$ such that
\be\label{sobo12}
\|u\, \cdot \,v|W^{m_2}_p(\Rd)\|\leq c \|u|W^{m_0}_p(\Rd)\|\, \|v|W^{m_1}_p(\Rd)\|, 
\ee
holds for all $u\in W^{m_0}_p(\Rd)$ and $v\in W^{m_1}_p(\Rd)$, where 
\be\label{sobo11}
 m_2 \le m_0+ m_1 -\frac dp\, .
\ee

\begin{prop}\label{mult113}
Let $1< p < \infty$,  $m_0, m_1 \in \N$, $m_2 \in \N_0$, $0 \le \ell < d$, and $a_0, a_1 \in \R$. 
Suppose \eqref{sobo10} and \eqref{sobo11}. We put  
\[
 a_2:= a_0 + a_1 - \frac{d}{p} \, .
\]
Then there exists a constant $c$ s.t.
	\begin{equation*}
		\|\, u\, \cdot \, v\,  |\calk^{m_2}_{a_2,p} (\R^d,\R^\ell_*)
			\|\le c \, \|\, u\, |\calk^{m_0}_{a_0,p} (\R^d,\R^\ell_*)\| \, \|\, v\,  |\calk^{m_1}_{a_1,p} (\R^d,\R^\ell_*)\|
	\end{equation*}
	holds for all $u \in   \calk^{m_0}_{a_0,p} (\R^d,\R^\ell_*)$ and all $v \in   \calk^{m_1}_{a_1,p} (\R^d,\R^\ell_*)$. 
\end{prop}

\begin{proof}
 The proof follows along  the same lines as the proof of Theorem \ref{mult111}.
\end{proof}

There is a  second possibility. Instead of a shift in the smoothness parameters as above, we shall now allow a shift in the integrability.
Again our approach will be based on Theorem 4.5.2 and Corollary 4.5.2 in \cite{RS}.
Let $1 <p< \infty$, $m \in \N$, and 
\be\label{sobo7}
2d\, \Big(\frac{1}{p}-\frac 12 \Big)  < m < \frac dp  \, .
\ee 
We put 
\be\label{sobo9}
t:= \frac{d}{2 \frac dp  -m}\, .
\ee
Then there exists a constant $c>0$ such that
\be\label{sobo6}
\|u\, \cdot \,v|W^{m}_{t}(\Rd)\|\leq c \|u|W^{m}_p(\Rd)\|\, \|v|W^{m}_p(\Rd)\| 
\ee
holds for all $u,v\in W^{m}_p(\Rd)$. It is easy to see that the left-hand side of  \eqref{sobo7} is equivalent to $t>1$.
The optimality of $t$ follows by employing the family $f_\alpha$.
Let $t < t'$ and put
\[
 \varepsilon := \frac d2 \Big(\frac 1t -\frac 1{t'} \Big).
\]
Then the function $f_\alpha$ with 
$\alpha = \frac dp - m-\varepsilon$ belongs to $W^m_p (\Rd)$ and 
$f_\alpha \, \cdot \, f_\alpha \not\in W^m_{t'}(\Rd)$ 
because of $2\alpha = \frac{d}{t'}-m$.

\begin{satz}\label{mult112}
Let $1< p < \infty$, $m \in \N$,  $0 \le \ell < d$, and $a_0 \in \R$. 
Let $t$ be defined as in \eqref{sobo9} and put
$ \tilde{a}_1:= 2 a_0-m$.
Suppose \eqref{sobo7}. Then for any $a_1 < \tilde{a}_1$ 
there exists a constant $c$ s.t.
\begin{equation*}
\|\, u\, \cdot \, v\,  |\calk^{m}_{a_1,t} (\R^d,\R^\ell_*)
\|\le c \, \|\, u\, |\calk^{m}_{a_0,p} (\R^d,\R^\ell_*)\| \, \|\, v\,  |\calk^{m}_{a_0,p} (\R^d,\R^\ell_*)\|
	\end{equation*}
	holds for all $u, v \in   \calk^{m}_{a_0,p} (\R^d,\R^\ell_*)$. 
\end{satz}

\begin{proof}
We follow the same strategy as used in the proof of Theorem \ref{mult1}.
Again it suffices to comment on the modifications needed in comparison with the proof of Theorem \ref{mult1}.
We need to modify \eqref{w-5} by using \eqref{sobo6}, which gives 
\begin{eqnarray}\label{w-8}
\|\, u\, \cdot \, v\, && \hspace{-0.7cm} |\calk^{m}_{a_1,t} (\R^d,\R^\ell_*)\|^t
\nonumber
\\
&\lesssim & \sum_{j=1}^\infty 2^{ja_1t} \, 2^{-dj}\,
\| \,u (2^{-j+1} \, \cdot \, )\, |W^{m}_p (D_1)\|^t \,
	\|\, v (2^{-j+1} \, \cdot \, )\, |W^{m}_p (D_1)\|^t\,.
\end{eqnarray}
Let $\varepsilon := \frac t2 \, (\tilde{a}_1 -a_1) >0$ and $a_2 := \frac {a_1}{2} + d (\frac 1p -\frac{1}{2t}) $. 
Observe $d (\frac 1p -\frac{1}{2t}) = \frac m2$.
It follows
\begin{align*}
\|\, u\, \cdot \, v\,  |\calk^{m}_{a_1,t} (\R^d,\R^\ell_*)\|^t
& \lesssim \sum_{j=1}^\infty 2^{ja_1t} \, 2^{-dj} \biggl(\sum_{|\beta| \le m} 
2^{-j|\beta|p} \, 2^{dj} \, \int_{D_j} | \, \partial^\beta u (x)  |^p dx\biggr)^{t/p}
\\
&\qquad\times \quad \biggl(\sum_{|\beta| \le m} 2^{-j|\beta|p} \, 2^{dj} \, 
\int_{D_j} | \, \partial^\beta v (x)  |^p\,  dx\biggr)^{t/p}
\\
& \lesssim \sum_{j=1}^\infty 2^{-j\varepsilon} \sup_{j\ge 0} 2^{j\varepsilon} \biggl(\sum_{|\beta| \le m} 
2^{-j|\beta|p} \, 2^{j(\frac {a_1}{2} + d (\frac 1p -\frac{1}{2t}))p} \, 
\int_{D_j}| \, \partial^\beta u (x)  |^p\, dx\biggr)^{t/p}
\\
& \qquad\times \quad \sup_{j\ge 0} \biggl(\sum_{|\beta| \le m} 
2^{-j|\beta|p} \, 2^{j(\frac {a_1}{2} + d (\frac 1p -\frac{1}{2t}))p} \, 
\int_{D_j} |\, \partial^\beta v (x)  |^p\,  dx\biggr)^{t/p}
\\
& \lesssim  \|\, u\, |\calk^{m}_{a_2+ \varepsilon /t,p} (\R^d,\R^\ell_*)\|^t\ 
\|\,  v\,  |\calk^{m}_{a_2,p} (\R^d,\R^\ell_*)\|^t\, .
\end{align*} 
Because of  $a_0= a_2+ \varepsilon /t$ and the monotonicity of Kondratiev spaces with respect to $a$, 
{i.e., $\calk^{m}_{a_2 +\frac{\varepsilon}{t},p} (\Omega,M)\hookrightarrow \calk^{m}_{a_2,p} (\Omega,M)$}, 
the claim follows.
\end{proof}

\begin{rem}
 \rm
We compare Theorem \ref{mult111} and Theorem \ref{mult112}.
Observe
\[
\calk^{m}_{2a_0-m,t} (\R^d,\R^\ell_*) \hookrightarrow \calk^{m_1}_{2(a_0-d/(2p)),p} (\R^d,\R^\ell_*) \, , 
\]
see Proposition \ref{lemma-embedding-sobolev}. However, again by Proposition \ref{lemma-embedding-sobolev}
it follows 
\[
\calk^{m}_{2a_0-m-\varepsilon,t} (\R^d,\R^\ell_*) \not\subset \calk^{m_1}_{2(a_0-d/(2p)),p} (\R^d,\R^\ell_*) \, , 
\]
for any $\varepsilon >0$. In addition, the assumptions in
Theorem \ref{mult111} and Theorem \ref{mult112} are different. 
\end{rem}

As already mentioned  in the introduction we  are interested in 
some semilinear elliptic PDEs with nonlinearity given by $u^n$ for some {$n>1$}.
For later use we now collect  the consequences of our previous results for the  mapping $u \mapsto u^n$.

\begin{cor}\label{power1}
Let  $m_0,m\in \N$, $m_1\in\N_0$,   $a\in\R$, and $0 \le \ell < d$.
\\
{\rm (i)}  Let either $1 < p< \infty$ and  $m>d/p$ or $p=1$ and $m\ge d$. Then we have for all natural numbers $n>1$, 
	\[
		u^n\in\calk^m_{na-\frac{d(n-1)}{p},p}(\Rd,\R^\ell_*)\quad\text{for all}\quad u\in\calk^m_{a,p}(\Rd,\R^\ell_*)\,,
	\]
	together with the estimate
	\[
		\|\, u^n\, |\calk^{m}_{na-\frac{d(n-1)}{p},p} (\Rd,\R^\ell_*)\|
			\leq c^{n-1}\, \|\, u\, |\calk^{m}_{a,p} (\Rd,\R^\ell_*)\|^n\,,
	\]
where $c$ is the constant in \eqref{ws-08}.
\\
{\rm (ii)} Let  $1 < p< \infty$ and  $d/(2p)\le m_0 <d/p$. 
Let $m_1 \le 2m_0 - d/p$. Then we have 
	\[
		u^2 \in \calk^{m_1}_{2a-\frac{d}{p},p}(\Rd,\R^\ell_*)\quad\text{for all}\quad 
		u\in\calk^{m_0}_{a,p}(\Rd,\R^\ell_*)\,,
	\]
	together with the estimate
	\[
		\|\, u^2\, |\calk^{m_1}_{2a-\frac{d}{p},p} (\Rd,\R^\ell_*)\|
			\leq c \, \|\, u\, |\calk^{m_0}_{a,p} (\Rd,\R^\ell_*)\|^2\,,
	\]
where $c$ is the constant in \eqref{sobo14}.
\\
{\rm (iii)} Let  $1 < p< \infty$, $t$ as in \eqref{sobo9} and  
\[
2d\, \Big(\frac{1}{p}-\frac 12 \Big)  < m < \frac dp  \, .
\]
Then, for any $a_1 < 2a-m$, there exists a constant $c$ such that 
	\[
		u^2 \in \calk^{m}_{a_1,t}(\Rd,\R^\ell_*)\quad\text{for all}\quad 
		u\in\calk^{m}_{a,p}(\Rd,\R^\ell_*)\,,
	\]
	together with the estimate
	\[
		\|\, u^2\, |\calk^{m}_{a_1,t} (\Rd,\R^\ell_*)\|
			\leq c \, \|\, u\, |\calk^{m}_{a,p} (\Rd,\R^\ell_*)\|^2\,.
	\]
\end{cor}

\begin{proof}
The result in (i) follows by induction upon applying Corollary \ref{lem-mult1}  to $u$ and $u^{n-1}$.
Parts (ii) and (iii) are just  special cases of Theorem \ref{mult111} and Theorem \ref{mult112}, respectively.
\end{proof}

\begin{rem}
\rm
Applying Proposition \ref{mult113} one can use the same induction argument as for the proof of part (i).
However, the results obtained in this way seem to make sense only in the very special situation that 
$d/p$ is a natural number.
The reason for that can be found in the restriction \eqref{sobo11}. 
If $d/p$ is not a natural number one is loosing too much regularity through the induction process.
For $d/p \in \{1, \ldots , d-1\}$ and $(n-1)\frac dp \le nm$ one obtains:
\[
u^n\in\calk^{nm- (n-1)\frac dp}_{na-(n-1)\frac{d}{p},p}(\Rd,\R^\ell_*)\quad\text{for all}\quad u\in\calk^m_{a,p}(\Rd,\R^\ell_*)\,,
\]
together with the estimate
\[
\|\, u^n\, |\calk^{nm-(n-1)\frac dp}_{na-(n-1)\frac{d}{p},p} (\Rd,\R^\ell_*)\|
\leq c^{n-1}\|\, u\, |\calk^{m}_{a,p} (\Rd,\R^\ell_*)\|^n\,. 
\]
\end{rem}


\subsubsection{Moser-type inequalities}


Moser \cite{Mo} was the first to see that one can improve the estimates of products in case one knows 
that both factors are bounded.
In fact the following is true: Let $1 < p<\infty$ and  $m\in\N$. Then $W^m_p (\Rd)$ is an algebra with respect to pointwise multiplication
and there exists a constant c such that 
\begin{eqnarray}\label{ws-30}
\|\, u\, \cdot \, v\,  |W^{m}_{p} (\Rd)\|
			 & \le &  c \, (\, \|\, u\, |W^{m}_{p} (\Rd)\| \, \| \, v \, |L_\infty (\Rd)\| 
\nonumber
\\
&& \qquad \qquad\qquad +\quad 
\, \|\, v\,  |W^{m}_{p} (\Rd)\|\, \| \, {u} \, |L_\infty (\Rd)\|)
	\end{eqnarray}
holds for all $u,v \in W_p^m (\Rd) \cap L_\infty (\Rd) $, see \cite[4.6.4]{RS}.
\\
Inserting this modification into Step 2 of the proof of Theorem \ref{mult1}  we obtain the following.

\begin{satz}\label{thm-mult11}
Let $1< p < \infty$, $m \in \N$, $0 \le \ell < d$, and $a \in \R$. 
\\
Then there exists a constant $c$ s.t.
\begin{eqnarray*}
\|\, u\, \cdot \, v\,  |\calk^{m}_{a,p} (\R^d,\R^\ell_*)\| & \le &  c \, 
(\|\, u\, |\calk^{m}_{a,p} (\R^d,\R^\ell_*)\| \, \| \, v \, |L_\infty (\Rd)\| 
\\
&& \qquad \qquad + \quad  
\|\, v\,  |\calk^{m}_{a,p} (\R^d,\R^\ell_*)\|\, \| \, u \, |L_\infty (\Rd)\|)
\end{eqnarray*}
	holds for all $u, v \in   \calk^{m}_{a,p} (\R^d,\R^\ell_*)\cap L_\infty (\Rd)$. 
\end{satz}


\subsubsection{Products of Kondratiev spaces with different $p$}


In the sequel we  are not interested in the most general situation, only in a few special cases.
Quite easy to prove is the following.

\begin{prop}\label{mult2a}
Let $1 \le  p <  \infty$, $m \in \N$, $0 \le \ell < d$, and $a \in \R$. Then there exists a constant $c$ s.t.
\begin{equation*}
		\|\, u\, \cdot \, v\,  |\calk^{m}_{a,p} (\R^d,\R^\ell_*)
			\|\le c \, \|\, v\, |\calk^{m}_{0,\infty} (\R^d,\R^\ell_*)\| \, \|\, u\,  |\calk^{m}_{a,p} (\R^d,\R^\ell_*)\|
	\end{equation*}
	holds for all $u \in   \calk^{m}_{a,p} (\R^d,\R^\ell_*)$ and $v \in   \calk^{m}_{0,\infty} (\R^d,\R^\ell_*)$. 
\end{prop}

\begin{proof}
We only need to modify several steps of the proof of  Theorem \ref{mult1}. 
Using Leibniz rule and the fact that the sets 
\[
D_j : =\{x\in \Rd: ~~2^{-j-1}<\rho(x)<2^{-j+1}\},\qquad j\in\N_0, 
\]
have finite overlap, we estimate
\begin{align*}
\intertext{$\|\, u\, \cdot \, v\, |\calk^{m}_{a,p} (\R^d,\R^\ell_*)\|^p$}
&\lesssim  \sum_{j=0}^{\infty}\sum_{|\alpha|\leq m}
\sum_{\alpha=\beta+\gamma}2^{-j(|\beta|-a)p}2^{-j|\gamma|p}\int_{D_j}|\partial^{\beta}u(x)\partial^{\gamma}v(x)|^p \, dx
\\
&\lesssim  \sum_{j=0}^{\infty}\sum_{|\beta|\leq m}\sum_{|\gamma|\leq m}2^{-j(|\beta|-a)p}2^{-j|\gamma|p}
\sup_{x\in D_j}|\partial^{\gamma}v(x)|^p\int_{D_j}|\partial^{\beta}u(x)|^p \, dx
\\
&\leq  \left(\sum_{|\gamma|\leq m}\sup_{j\geq 0}\sup_{x\in D_j}|2^{-j|\gamma|}\partial^{\gamma}v(x)|^p\right)	
	\sum_{j=0}^{\infty}\sum_{|\beta|\leq m} 2^{-j(|\beta|-a)p}	\int_{D_j}|\partial^{\beta}u(x)|^p \, dx
\\
&\lesssim  \left(\sum_{|\gamma|\leq m}\sup_{j\geq 0}\sup_{x\in D_j}|\rho(x)^{|\gamma|}\partial^{\gamma}v(x)|	\right)^p	\|u|\calk^{m}_{a,p} (\R^d,\R^\ell_*)\|^p
\\
&\lesssim  \|v|\calk^{m}_{0,\infty} (\R^d,\R^\ell_*)\|^p\|u|\calk^{m}_{a,p} (\R^d,\R^\ell_*)\|^p\, .
\end{align*}
This proves the claim.		
\end{proof}

In our intended application in \cite{DHSS1}  we can also allow for a shift in the regularity 
parameter $m$ in the pointwise multiplier assertion, i.e., for the product $u\, \cdot \, v$ we 
require less weak derivatives than for $u$ and/or $v$. This aspect leads to the following 
modification.

\begin{cor}\label{mult2}
Let  $m,n \in \N$ and  $0 \le \ell < d$.
\\
{\rm (i)} Let $ \max(1,d/n) < p < \infty$ and $a\geq\frac{d}{p}-n$. Then there exists a constant $c$ s.t.
\begin{equation}\label{eq1}
\|\, u\, \cdot \, v\,  |\calk^{m}_{a,p} (\R^d,\R^\ell_*) \|
\le c \, \|\, v\, |\calk^{m+n}_{a+n,p} (\R^d,\R^\ell_*)\| \, \|\, u\,  |\calk^{m}_{a,p} (\R^d,\R^\ell_*)\|
\end{equation}
holds for all $u \in   \calk^{m}_{a,p} (\R^d,\R^\ell_*)$ and $v \in   \calk^{m+n}_{a+n,p} (\R^d,\R^\ell_*)$. 
\\
{\rm (ii)} Let $p = 1$,   $ d \le n$, and $a\geq {d}-n$. Then 
\eqref{eq1} is true as well.  	
\end{cor}

\begin{proof}
In view of Proposition \ref{mult2a} it is enough to apply 
$\calk^{m+n}_{a+n,p}(\R^d,\R^\ell_*) \hookrightarrow \calk^{m}_{0,\infty} (\R^d,\R^\ell_*)$, which by  Corollary \ref{sobolevinfty} is the case if    
\[
m+n -\frac dp>	m, \qquad \text{i.e.,}\qquad n >\frac dp,\qquad \text{and}\qquad a+n -\frac dp\geq 0. 
\]
By our assumptions on $p$ and $a$ the proof is complete. 	
\end{proof}

\begin{rem}
 \rm
(i)
Observe, that in Corollary \ref{mult2} we do not need that $m$ is large. We only use that $a$ is sufficiently large. 
\\
(ii) A simple reformulation yields in case $n=2$
\begin{equation}\label{eq2}
		\|\, u\,\cdot \,  v\,  |\calk^{m-1}_{a-1,p} (\R^d,\R^\ell_*)
			\|\le c \, \|\, v\, |\calk^{m+1}_{a+1,p} (\R^d,\R^\ell_*)\| \, \|\, u\,  |\calk^{m-1}_{a-1,p} (\R^d,\R^\ell_*)\|
\end{equation}
if $m \in \N$,   $\max (1,d/2) <p< \infty $, and $a\ge \frac dp -1$.
\end{rem}



\subsection{Pointwise multiplication in Kondratiev spaces on domains of generalized polyhedral type}


In order to shift the results on multiplication  obtained for $\calk^{m}_{a,p} (\R^d,\R^\ell_*)$ to 
$\calk^{m}_{a,p} (D,M)$ with $(D,M)$ being a pair of generalized polyhedral type, we proceed as in case of the continuous embeddings.
As a first step we employ Proposition \ref{extension} and conclude that all sufficient conditions in  Subsection \ref{pm1}
remain sufficient conditions for Kondratiev spaces on pairs $(D,M)$ as in Cases I-III.
Next we use Lemma \ref{deco2} to show the same for Kondratiev spaces on polyhedral cones.
Now we are ready to prove sufficiency also for Kondratiev spaces on domains of  polyhedral type by making use of  Lemma \ref{deco}.
Finally, we apply Lemma \ref{deco3} to establish sufficiency for pairs of generalized  polyhedral type.

\begin{satz}\label{very general}
Theorem \ref{mult1}, Corollary \ref{mult11}(i),  Corollary \ref{lem-mult1}, Theorem \ref{mult111}, Proposition \ref{mult113},
Theorem \ref{mult112}, Corollary \ref{power1}, Theorem \ref{thm-mult11}, Proposition \ref{mult2a}, and 
Corollary  \ref{mult2} carry over to Kondratiev spaces with respect to pairs $(D,M) $ of generalized polyhedral type. 
\end{satz}

Also Corollary \ref{mult11}(ii) has a counterpart, but restricted to spaces on smooth cones.

\begin{cor}\label{mult1111}
Let $1\le p < \infty$, $m \in \N$, and  $a \in \R$. Furthermore, let $K$ be the smooth cone from Case I.
\\
The Kondratiev space $\calk^{m}_{a,p} (K,\{0\})$ is an algebra with respect to pointwise multiplication
if, and only if, $a\ge\frac{d}{p}$ with   either $1<p<\infty$ and  $m>d/p$ or $p=1$ and $m\ge d$. 
\end{cor}


\section{Appendix - concrete examples}


The following three lemmas are based on straightforward
elementary calculations. Therefore all details are omitted. 

\begin{lem}\label{eins}
 Let $1\le p \le \infty$, $a\in \R$, and $m \in \N_0$.
 \\
 (i) Let $K$ be a smooth cone as in Case I.
 Then $1 \in \calk^{m}_{a,p}(K,\{0\})$ if, and only if, either $a < d/p$ for $p < \infty$ or if 
 $a\le 0$ for $p=\infty$.
 \\
 (ii) Let $P$ be a specific nonsmooth cone as in Case II.
 Then $1 \in \calk^{m}_{a,p}(P,\Gamma)$ if, and only if, either $a < (d-1)/p$ for $p < \infty$ or if 
 $a\le 0$ for $p=\infty$.
 \\
 (iii) Let $I$ be a specific dihedral domain as in Case III.
 Then $1 \in \calk^{m}_{a,p}(I,M_\ell)$ if, and only if, either $a < (d-\ell)/p$ for $p < \infty$ or if 
 $a\le 0$ for $p=\infty$.
 \\
 (iv) Let $Q$ be a polyhedral cone  as in Case IV.
 Then $1 \in \calk^{m}_{a,p}(Q,M)$ if, and only if, either $a < 2/p$ for $p < \infty$ or if 
 $a\le 0$ for $p=\infty$.
\end{lem}

\begin{rem}
 \rm
 The case of a smooth cone is considered in \cite[Lemma 2.5]{CDN}.
\end{rem}

Let $\tilde{\varrho}$ denote the regularized distance function. Then, for any $b \in \R$, the mapping 
$T_b : ~ u \mapsto \tilde{\varrho}^b \, u $ yields an isomorphism  of $\calk^m_{a,p}(\Omega,M)$ onto $\calk^m_{a+b,p}(\Omega,M)$. 
As an immediate conclusion 
of Lemma \ref{eins} we obtain the following.

\begin{lem}\label{zwei}
 Let $1\le p \le \infty$, $a, b\in \R$, $0 \le l<d $, and $m \in \N_0$.
 \\
 (i) Let $K$ be a smooth cone as in Case I.
 Then $\tilde{\varrho}^b \in \calk^{m}_{a+b,p}(K,\{0\})$ if, and only if, either $a < d/p$ for $p < \infty$ or if 
 $a\le 0$ for $p=\infty$.
 \\
 (ii) Let $P$ be a specific nonsmooth cone as in Case II.
 Then $\tilde{\varrho}^b \in \calk^{m}_{a+b,p}(P,\Gamma)$ if, and only if, either $a < (d-1)/p$ for $p < \infty$ or if 
 $a\le 0$ for $p=\infty$.
 \\
 (iii) Let $I$ be a specific dihedral domain as in Case III.
 Then $\tilde{\varrho}^b \in \calk^{m}_{a+b,p}(I,M_\ell)$ if, and only if, either $a < (d-\ell)/p$ for $p < \infty$ or if 
 $a\le 0$ for $p=\infty$.
 \\
 (iv) Let $Q$ be a polyhedral cone  as in Case IV.
 Then $\tilde{\varrho}^b \in \calk^{m}_{a+b,p}(Q,M)$ if, and only if, either $a < 2/p$ for $p < \infty$ or if 
 $a\le 0$ for $p=\infty$.
\end{lem}

In our model case $(\Rd, \R^{\ell}_\ast)$ we need a modification.
Let $\psi $ denote a smooth cut-off function, i.e., $\psi (x)=1$, $|x|\le 1$, and $ \psi (x)=0 $ if $|x|\ge 3/2$.

\begin{lem}\label{drei}
 Let $1\le p \le \infty$, $a, b\in \R$, $0 \le \ell <d$, and $m \in \N_0$.
   Then $\tilde{\varrho}^b\, \cdot \, \psi \in \calk^{m}_{a+b,p}(\Rd, \R^{\ell}_\ast)$ if, and only if, either $a < (d-\ell)/p$ for $p < \infty$ or if 
 $a\le 0$ for $p=\infty$.
\end{lem}


\bigskip
\vbox{\noindent Stephan Dahlke\\
Philipps-Universit\"at Marburg\\
FB12 Mathematik und Informatik\\
Hans-Meerwein Stra\ss e, Lahnberge\\
35032 Marburg\\
Germany\\
E-mail: {\tt dahlke@mathematik.uni-marburg.de}\\
WWW: {\tt http://www.mathematik.uni-marburg.de/$\sim$dahlke/}\\}

\bigskip
\smallskip
\vbox{\noindent  Markus Hansen\\
Philipps-Universit\"at Marburg\\
FB12 Mathematik und Informatik\\
Hans-Meerwein Stra\ss e, Lahnberge\\
35032 Marburg\\
Germany\\
E-mail: {\tt hansen@mathematik.uni-marburg.de}\\
WWW: {\tt http://www.mathematik.uni-marburg.de/$\sim$hansen/}\\}

\bigskip
\smallskip
\vbox{\noindent  Cornelia Schneider\\
Friedrich-Alexander-Universit\"at Erlangen-N\"urnberg\\ 
Department Mathematik, AM3\\
Cauerstr. 11\\ 
91058 Erlangen \\ 
Germany\\
E-mail: {\tt schneider@math.fau.de}\\
WWW: {\tt http://www.math.fau.de/$\sim$schneider/}\\}

\bigskip
\smallskip
\vbox{\noindent  Winfried Sickel\\
Friedrich-Schiller-Universit\"at Jena\\ 
Mathematisches Institut\\
Ernst-Abbe-Platz 2\\ 
07743 Jena \\ 
Germany\\
E-mail: {\tt winfried.sickel@uni-jena.de}\\
WWW: {\tt https://users.fmi.uni-jena.de/$\sim$sickel/}\\}

\bigskip

\end{document}